\documentclass[11pt,reqno]{amsart}
\usepackage{amssymb}
\usepackage{amsfonts}
\usepackage{mathrsfs}
\usepackage{amsmath}
\usepackage{graphicx}
\usepackage{hyperref}
\usepackage{float}
\usepackage{epstopdf}
\usepackage{color}
\usepackage{bm}
\usepackage{comment}
\usepackage{soul}

\allowdisplaybreaks
\newtheorem{theorem}{Theorem}[section]
\newtheorem{proposition}[theorem]{Proposition}
\newtheorem{lemma}[theorem]{Lemma}

\newtheorem{definition}[theorem]{Definition}

\def\diam{\mathrm{diam}}

\def\SG{\mathcal{SG}}

\def\mcE{\mathcal{E}}

\numberwithin{equation}{section}

\begin{document}
\title[Equivalence of Besov spaces on p.c.f. self-similar sets]{Equivalence of Besov spaces on p.c.f. self-similar sets}

\author{Shiping Cao}
\address{Department of Mathematics, Cornell University, Ithaca 14853, USA}
\email{sc2873@cornell.edu}
\thanks{}

\author{Hua Qiu}
\address{Department of Mathematics, Nanjing University, Nanjing 210093, China}
\email{huaqiu@nju.edu.cn}
\thanks{The research of Qiu was supported by the NSFC grant 11471157}


\date{}


\begin{abstract}
On p.c.f. self-similar sets, of which the walk dimensions of heat kernels are in general larger than 2, we find a sharp region where two classes of Besov spaces, the heat Besov spaces $B^{p,q}_\sigma(K)$ and the Lipschitz-Besov spaces $\Lambda^{p,q}_\sigma(K)$, are identitical. In particular, we provide concrete examples that $B^{p,q}_\sigma(K)=\Lambda^{p,q}_\sigma(K)$ with $\sigma>1$. Our method is purely analytical, and does not involve any heat kernel estimate. 
\end{abstract}
\maketitle

\tableofcontents

\section{Introduction}\label{intro}
In this paper, we study the identity of two classes of Besov spaces on post-critically finite (p.c.f.) self-similar sets. One class is the heat Besov spaces $B^{p,q}_\sigma(K)$, defined with the Neumann Laplacian $\Delta_N$, which was introduced in the pioneering study of Brownian motions on typical fractals \cite{BB,BBKT,BP,G,KZ,Lindstrom}, and was later constructed in a purely analytical way by J. Kigami \cite{ki1,ki2} on general p.c.f. self-similar sets. The heat Besov spaces $B^{p,q}_\sigma(K)$ are defined as potential spaces following \cite{HZ}, 
\[B^{p,q}_\sigma(K)=\Big\{f\in L^p(K):\Big(\int_{0}^\infty \big(t^{-\sigma/2}\big\|(t\Delta_N)^kP_tf\big\|_{L^p(K)}\big)^{q}dt/t\Big)^{1/q}<\infty\Big\},\] 
where $\{P_t\}_{t\geq 0}$ is the heat semigroup associated with $\Delta_N$. Here we take the measure $\mu$ to be self-similar and $d_H$-regular with the effective resistance metric $R(\cdot,\cdot)$ on $K$, where $d_H$ is the Hausdorff dimension of $K$ under $R$.
The other class $\Lambda^{p,q}_\sigma(K)$, named Lipschitz-Besov spaces, is defined directly with integration of difference of functions,
\[\Lambda^{p,q}_\sigma(K)=\Big\{f\in L^p(K):\Big(\int_0^\infty\big(\int_Kt^{-d_H}\int_{B_t(x)}\frac{|f(x)-f(y)|^p}{t^{\sigma pd_W /2}}d\mu(y)d\mu(x)\big)^{q/p}\frac{dt}{t}\Big)^{1/q}<\infty\Big\},\] 
where $B_t(x)$ is the ball of radius $t$ centered at $x$ under the metric $R$, and $d_W=1+d_H$ is the walk dimension of the associated heat kernel. Roughly speaking, $d_H$ reflects the growth of the measure, and $d_W$ reflects the speed of the diffusion process. More explanations on general metric measure spaces can be found in \cite{GHL}. 

The relationship between the two classes of Besov spaces $B^{p,q}_\sigma(K)$ and $\Lambda^{p,q}_\sigma(K)$ has been a long term problem \cite{P} on general metric measure spaces, and the identity
\begin{equation}\label{eqn11}
B^{p,q}_\sigma(K)=\Lambda^{p,q}_\sigma(K)
\end{equation}
is of particular interest. For $p=q=2$ and $0<\sigma<1$, when the Besov spaces coincide with the Sobolev spaces, under some weak assumption of heat kernel estimates, Hu and Z\"{a}hle \cite{HZ} showed that (\ref{eqn11}) holds, as well as Strichartz \cite{s2} obtained the same result on products of p.c.f. self-similar sets at the same time. Later, Grigor'yan and Liu proved that (\ref{eqn11}) holds for any $1<p,q<\infty$ and any $0<\sigma<\frac{2\Theta}{d_W}\wedge 1$, where $\Theta$ denotes the H\"{o}lder exponent of the heat kernel (see \cite{GL}). In particular, on p.c.f. self-similar sets, due to the sub-Gaussian heat kernel estimates \cite{HK,KS}, the existence of small H\"{o}lder exponent $\Theta$ was shown in \cite{GHL}. Until now, a larger region where (\ref{eqn11}) holds or not  is still hard to reach. 

Recently, Cao and Grigor'yan \cite{CG1,CG2} have made much progress showing (\ref{eqn11}) holds on a larger region, under the assumption of Gaussian heat kernel estimates. Their work introduces some new techniques, but the results and ideas are restricted to the  $d_W=2$ case. It is believed by the authors that more surprising and interesting phenomena about (\ref{eqn11}) are waited to be discovered for the  $d_W>2$ case. 

In this paper, we will focus on the p.c.f. self-similar sets, which are a class of well-known fractals where sub-Gaussian heat kernel estimates hold. In particular, we will describe a \textbf{sharp region} where (\ref{eqn11}) holds on p.c.f. self-similar sets. See the left picture of Figure \ref{equiarea}.

More precisely, we introduce a critical curve $\mathscr{C}$ for $1\leq p\leq \infty$,
\[\mathscr{C}(p)=\sup\big\{\sigma>0:\mathcal{H}_0\subset \Lambda^{p,\infty}_\sigma(K)\big\},\]
where $\mathcal{H}_0$ is the space of harmonic functions. We will prove the following theorem.

\begin{theorem}\label{thm11}
For $1<p<\infty$, $1\leq q\leq \infty$ and $0<\sigma<\mathscr{C}(p)$, we have $B^{p,q}_\sigma(K)=\Lambda^{p,q}_\sigma(K)$ with the equivalent norms. 
\end{theorem}

The identical region is sharp since locally harmonic functions (i.e., functions that are harmonic in a cell, and smooth elsewhere) always exist in Besov spaces $B^{p,q}_\sigma(K)$  for any $1<p,q<\infty$ and $\sigma>0$. Readers may compare $\mathscr{C}(p)$ with another important critical exponent 
\[\lambda^\#_p=\inf_{\sigma>0}\big\{\Lambda^{p,\infty}_\sigma(K)=constants\big\}.\]
Though, for $p=2$, we always have $\lambda^\#_2=1=\mathscr{C}(2)$ \cite{GHL}, we have to say that $\lambda^\#_p$ is not in general equal to $\mathscr{C}(p)$. In fact, it has been shown that $\lambda^\#_1=d_S$ in \cite{BV3} for nested fractals, while on the Sierpinski gasket, we can see $\mathscr{C}(1)<d_S$ with easy estimate (Example 3 in Section 3). 

\begin{figure}[htp]
	\includegraphics[width=5cm]{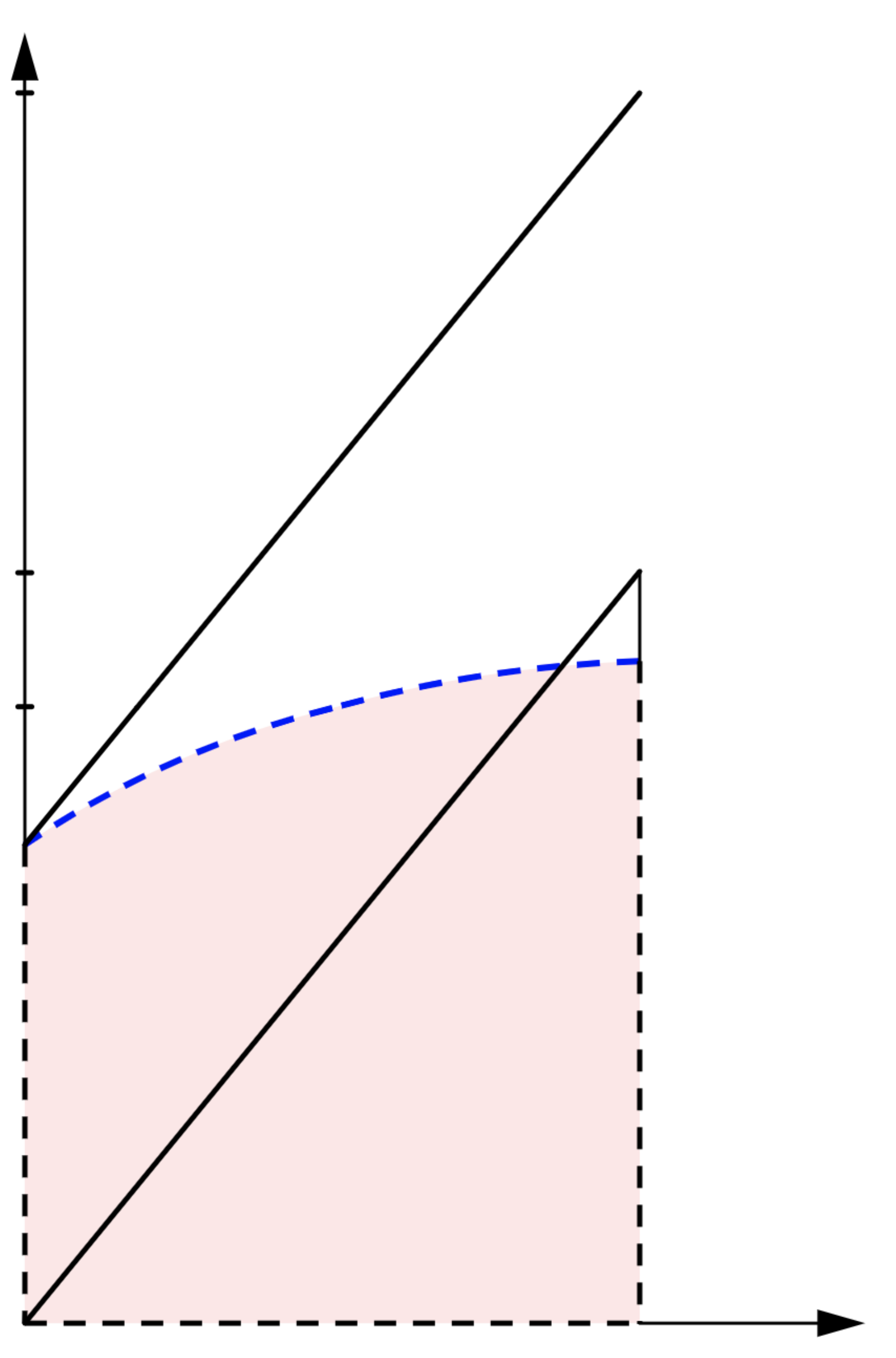}\qquad
	\includegraphics[width=5cm]{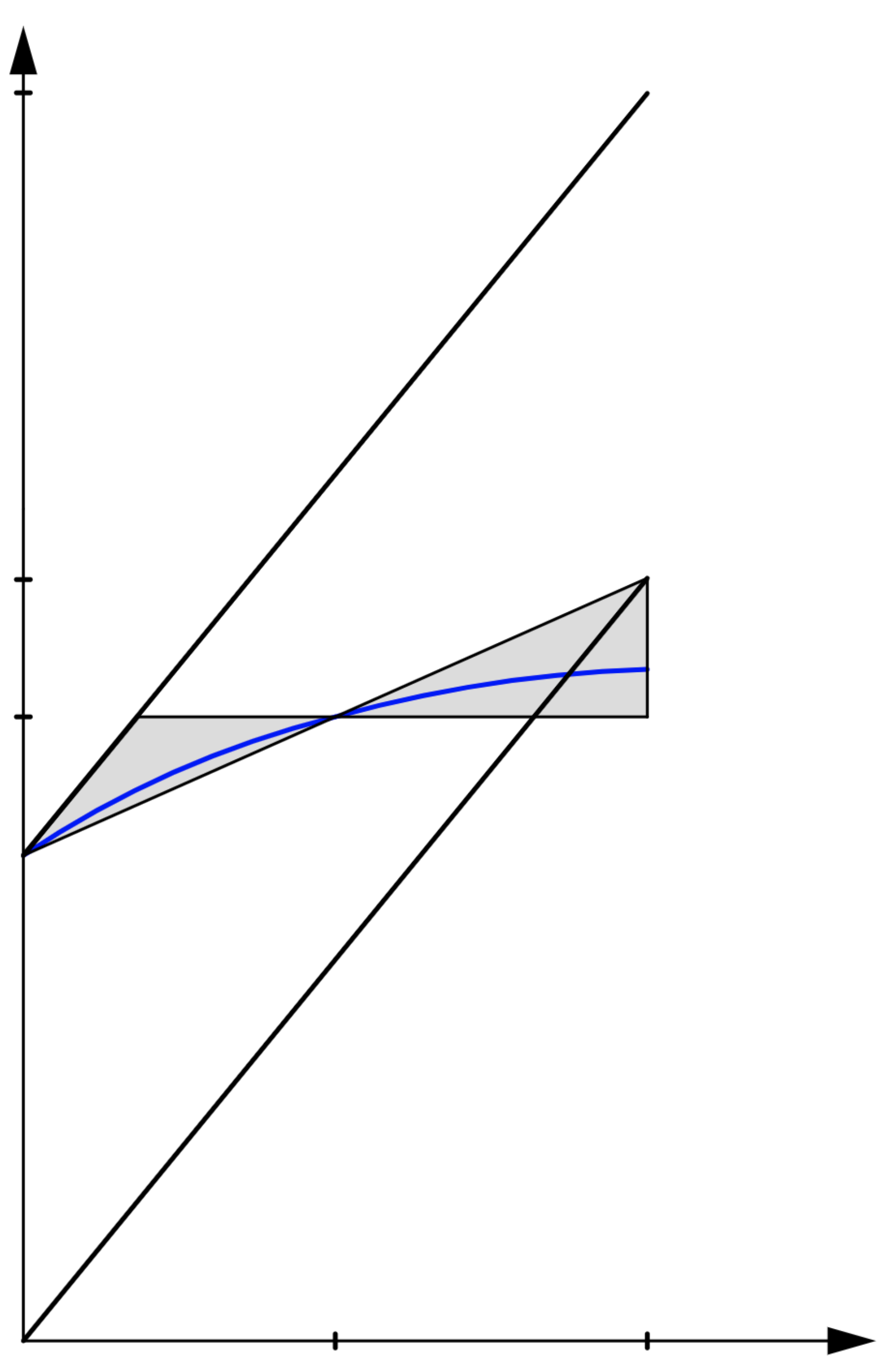}
	\begin{picture}(0,0)
	\put(-150,212){$\sigma$}	\put(-314,208){$\sigma$}
	\put(-151,101){$1$}\put(-316,99){$1$}	\put(-151,201){$2$}
	\put(-155,122){$d_S$}	\put(-320,122){$d_S$}
	\put(-157,78){$\frac{2}{d_W}$}\put(-322,78){$\frac{2}{d_W}$}
	\put(-190,14){$\frac 1p$}\put(-211,-5){$1$}\put(-205,107){$\mathscr C$}
	\put(-26,14){$\frac 1p$}\put(-40,108){$\mathscr C$}\put(-84,70){$\mathscr L_1$}\put(-84,148){$\mathscr L_2$}
	\put(-45,-6){$1$}\put(-95,-8){$\frac{1}{2}$}
	\end{picture}
	\caption{The sharp region for (\ref{eqn11}) and the possible area where $\mathscr{C}$ lies.}
	\label{equiarea}
\end{figure}
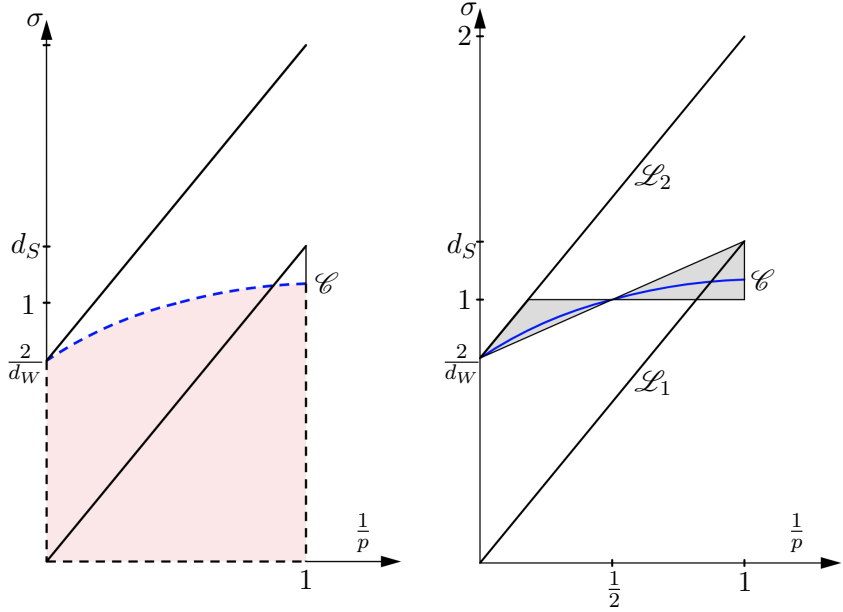

We will describe a narrow region where $\mathscr{C}$ lives in Proposition \ref{prop33}. Write 
\[\mathscr{L}_1(p)=\frac{d_S}{p},\quad \mathscr{L}_2(p)=2-\frac{d_S}{p'},\]
with $p'=\frac {p}{p-1}$ and $d_S=\frac{2d_H}{d_W}$ being the spectral dimension of the Laplacian\cite{KL}. $\mathscr{L}_1$ is naturally the critical line concerning the continuity of functions, and $\mathscr{L}_2$ is the critical line concerning the H\"{o}lder continuity of functions and thus the existence of normal derivatives at boundaries. In the authors' related works \cite{cq1,cq2,cq3}, there is a discussion on the role of these critical lines concerning the relationship between Sobolev spaces and (heat) Besov spaces on p.c.f. self-similar sets with different boundary conditions. In Proposition \ref{prop33}, we will show that the curve $\mathscr C$ is concave and increasing w.r.t. $\frac 1 p$, and in addition, \vspace{0.2cm} 

\textit{1). for $1\leq p\leq 2$, $1\leq\mathscr{C}(p)<\frac{2}{d_W}+\frac{2}{p}\cdot\frac{d_H-1}{d_W}$,}

\textit{2). for $2\leq p\leq \infty$,  $\frac{2}{d_W}+\frac{2}{p}\cdot\frac{d_H-1}{d_W}\leq\mathscr{C}(p)<1\wedge \mathscr{L}_2(p)$. }\vspace{0.2cm} 

\noindent See the right picture of Figure \ref{equiarea} for an illustration. In particular, it may happen that $\mathscr{C}(1)>1$ (for example, it is true for the Vicsek set and the Sierpinski gasket in standard setting), so (\ref{eqn11}) even holds in some cases when $\sigma>1$. This is a surprising result which was not mentioned in previous studies. 

The exact description of the critical curve $\mathscr{C}$, and the problem of whether the identity (\ref{eqn11}) holds along $\mathscr{C}$, are still out of reach, and are left to the future study. It is of particular interest to see whether $\mathscr{C}(1)>1$ always holds when $d_W>2$. Despite of this, we are able to fully describe the curve $\mathscr{C}$ for the class of Vicsek sets, see Example 2 in Section 3. 

At the end of this section, we mention that, throughout our study, the two critical lines $\mathscr L_1$ and $\mathscr L_2$ will play crucial roles as well. 

Now we briefly introduce the structure of this paper. Section 2 will serve as the background of this paper, where we introduce necessary knowledge and notations, including the p.c.f. self-similar sets, the Dirichlet forms and Laplacians on fractals, and the definitions of function spaces we consider here. In Section 3, we will discuss the critical curve $\mathscr{C}$ and provide several examples. This will help readers to understand the sharp region in the main theorem. In Section 4, we focus on the Lipschitz-Besov spaces $\Lambda^{p,q}_\sigma(K)$. We will provide two kinds of discrete type characterizations  of $\Lambda^{p,q}_\sigma(K)$, which will serve as a main tool towards the main theorem. In Section 5 and 6, we prove the main theorem, Theorem \ref{thm11}. In particular, we will show that $\Lambda^{p,q}_\sigma(K)\subset B^{p,q}_\sigma(K)$ for any $1<p<\infty$, $1\leq q\leq \infty$ and $0<\sigma<2$ in Section 5. In Section 6, we will prove the other direction, i.e. $B^{p,q}_\sigma(K)\subset \Lambda^{p,q}_\sigma(K)$ with $1<p<\infty$, $1\leq q\leq \infty$ and $0<\sigma<\mathscr{C}(p)$. 

Throughout the paper, we will always write $f\lesssim g$ if there is a constant $C>0$ such that $f\leq C g$ when we do not emphasize the constant $C$. In addition, we write $f\asymp g$ if both $f\lesssim g$ and $g\lesssim f$ hold.

\section{Preliminary}
The analysis on p.c.f. self-similar sets was originally developed by Kigami in \cite{ki2,ki3}. For convenience of readers, in this section,  first we will briefly recall the constructions of Dirichlet forms and Laplacians on p.c.f.  fractals. We refer to books \cite{ki3,s3} for details. Then we will provide the definitions of the two classes of Besov spaces, $B^{p,q}_\sigma(K)$ and $\Lambda^{p,q}_\sigma(K)$. There is a large literature on function spaces on fractals or on more general metric measure spaces, see  \cite{BV1,BV2,BV3,cq1,cq2,cq3,GKS,Gr,HKM,s1} and the references therein.

Let $\{F_i\}_{i=1}^N$ be a finite collection of contractions on a complete metric space $(\mathcal M,d)$. The self-similar set associated with the \textit{iterated function system (i.f.s.)} $\{F_i\}_{i=1}^N$ is the unique compact set $K\subset \mathcal M$ satisfying
\[K=\bigcup_{i=1}^N F_iK.\]
For $m\geq 1$, we define $W_m=\{1,\cdots,N\}^m$ the collection of \textit{words} of length $m$, and for each $w\in W_m$, denote
\[F_w=F_{w_1}\circ F_{w_2}\circ\cdots \circ F_{w_m}.\]
Set $W_0=\emptyset$, and let $W_*=\bigcup_{m\geq 0} W_m$ be the collection of all finite words. For $w=w_1w_2\cdots w_m\in W_*\setminus W_0$, we write $w^*=w_1w_2\cdots w_{m-1}$ by deleting the last letter of $w$.

Define the shift space $\Sigma=\{1,2,\cdots,N\}^{\mathbb{N}}$. There is a continuous surjection $\pi: \Sigma\rightarrow K$ defined by
\[\pi(\omega)=\bigcap_{m\geq 1}F_{[\omega]_m}K,\]
where for $\omega=\omega_1\omega_2\cdots$ in $\Sigma$ we write $[\omega]_m=\omega_1\omega_2\cdots \omega_m\in W_m$ for each $m\geq 1$. Let
\[C_K=\bigcup_{i\neq j}F_iK\cap F_jK,\quad \mathcal{C}=\pi^{-1}(C_K),\quad \mathcal{P}=\bigcup_{n\geq 1}\sigma^n \mathcal{C},\]
where $\sigma$ is the shift map define as $\sigma(\omega_1\omega_2\cdots)=\omega_2\omega_3\cdots$. $\mathcal{P}$ is called the \textit{post-critical set}. Call $K$  a \textit{post-critically finite (p.c.f.) self-similar set} if $\#\mathcal{P}<\infty$. In what follows, we always assume that $K$ is a connected p.c.f. self-similar set.

Let $V_0=\pi(\mathcal{P})$ and call it the \textit{boundary} of $K$. For $m\geq 1$, we always have $F_w K\cap F_{w'}K\subset F_w V_0\cap F_{w'}V_0$ for any $w\neq w'\in W_m$. Denote $V_m=\bigcup_{w\in W_m}F_wV_0$ and let $l(V_m)=\{f: f \text{ maps } V_m \text{ into } \mathbb{C}\}$. Write $V_*=\bigcup_{m\geq 0}V_m$.

Let $H=(H_{pq})_{p,q\in V_0}$ be a symmetric linear operator(matrix). $H$ is called a \textit{(discrete) Laplacian} on $V_0$ if  $H$ is non-positive definite; $Hu=0$ if and only if $u$ is constant on $V_0$; and $H_{pq}\geq 0$ for any $p\neq q\in V_0$.
Given a Laplacian $H$ on $V_0$ and a vector $\bm{r}=\{r_i\}_{i=1}^N$ with $r_i>0$, $1\leq i\leq N$, define the \textit{(discrete) energy form} on $V_0$ by
$$\mathcal{E}_0(f,g)=-(f,Hg),\quad \forall f,g\in l(V_0),$$
and inductively {{on $V_m$ by}}
$$\mathcal{E}_m(f,g)=\sum_{i=1}^Nr^{-1}_i\mathcal{E}_{m-1}(f\circ F_i, g\circ F_i), \quad \forall f,g\in l(V_m),$$
for $m\geq 1$. Write $\mathcal{E}_m(f,f)=\mathcal{E}_m(f)$ for short.

Say $(H,\bm{r})$ is a \textit{harmonic structure} if for any $f\in l(V_0)$,
\[\mathcal{E}_0(f)=\min\{\mathcal{E}_1(g):g\in l(V_1),g|_{V_{0}}=f\}.\]
In this paper, we will always assume that there exists a harmonic structure associated with $K$, and in addition, $0<r_i<1$ for all $1\leq i\leq N$. Call $(H,\bm{r})$ a \textit{regular harmonic structure} on $K$.

Now for each $f\in C(K)$, the sequence $\{\mathcal{E}_m(f)\}_{m\geq 0}$ is nondecreasing. Let
\begin{equation*}
\mathcal{E}(f,g)=\lim_{m\to\infty} \mathcal{E}_m(f,g) \text{ and }
dom\mathcal{E}=\big\{f\in C(K):\mathcal{E}(f)<\infty\big\},
\end{equation*}
where $f,g\in C(K)$ and  we write $\mathcal{E}(f):=\mathcal{E}(f,f)$ for short. Call $\mathcal{E}(f)$ the \textit{energy} of $f$.
It is known that $(\mathcal{E},dom\mathcal{E})$ turns out to be a local regular Dirichlet form on $L^2(K,\mu)$ for any Radon measure $\mu$ on $K$. 

An important feature of the form $(\mathcal{E},dom\mathcal{E})$ is the \textit{self-similar identity},
\begin{equation}\label{eq21}
\mathcal{E}(f,g)=\sum_{i=1}^Nr_i^{-1}\mathcal{E}(f\circ F_i, g\circ F_i), \quad \forall f,g\in dom\mathcal{E}.
\end{equation}
Furthermore, denote $r_w=r_{w_1}r_{w_2}\cdots r_{w_m}$ for each $w\in W_m, m\geq 0$. Then for  $m\geq 1$, we have
\begin{equation*}
\mathcal{E}_m(f,g)=\sum_{w\in W_m} r_w^{-1}\mathcal{E}_0(f\circ F_w, g\circ F_w),\quad \mathcal{E}(f,g)=\sum_{w\in W_m} r_w^{-1}\mathcal{E}(f\circ F_w, g\circ F_w).
\end{equation*}

\subsection{The Laplacian and harmonic functions}
To study the Besov spaces on $K$, we need a suitable metric and a comparable measure. Instead of the original metric $d$, a natural choice of metric is the effective resistance metric $R(\cdot,\cdot)$ \cite{ki3}, which matches the form $(\mathcal{E},dom\mcE)$.

\begin{definition}\label{def21}
	For $x,y\in K$, the {\em effective resistance metric} $R(x,y)$ between $x$ and $y$ is defined by
	\[R(x,y)^{-1}=\min\big\{\mathcal{E}(f):f\in dom\mcE,f(x)=0,f(y)=1\big\}.\]
\end{definition}

It is known that $R$ is indeed a metric on $K$ which is topologically equivalent to the metric $d$, and for each $w\in W_*$, we always have $diam(F_wK)\asymp r_w$, where $diam(F_wK)=\max\big\{R(x,y):x,y\in F_wK\big\}$. For convenience, we normalize $diam K$ to be $1$ and so that we additionally have $diam (F_wK)\leq r_w$, $\forall w\in W_*$. For $x\in K$ and $t>0$, we will use $B_t(x)$ to denote a ball centered at $x$ with radius $t$ in the sense of metric $R$. 

We will always choose the following self-similar measure $\mu$ on $K$.

\begin{definition}\label{def22}
	Let $\mu$ be the unique self-similar measure on $K$ satisfying
	\[\mu=\sum_{i=1}^N r_i^{d_H}\mu\circ F_i^{-1},\]
	and $\mu(K)=1$, where $d_H$ is determined by the equation
	$\sum_{i=1}^N r_i^{d_H}=1.$
\end{definition}

Clearly, we have $\mu(F_wK)=\mu_w:=\mu_{w_1}\mu_{w_2}\cdots\mu_{w_m}$ for any $m\geq 0,w\in W_m$. In addition, it is well-known that 
\[C^{-1}t^{d_H}\leq \mu\big(B_t(x)\big)\leq Ct^{d_H},\]
with some constant $C$ independent of $x,t$. 

With the Dirichlet form $(\mathcal{E},dom\mathcal{E})$ and the self-similar measure $\mu$, we can define the associated Laplacian on $K$ with the weak formula. 

\begin{definition}\label{def23}
	(a). Let $dom_0\mathcal{E}=\{\varphi\in dom\mathcal{E}:\varphi|_{V_0}=0\}$. For $f\in dom\mathcal{E}$, say $\Delta f=u$ if
	\[\mathcal{E}(f,\varphi)=-\int_K u\varphi d\mu , \quad\forall \varphi\in dom_0\mathcal{E}.\]
	
	(b). In addition, say $\Delta_N f=u$ if
	\[\mathcal{E}(f,\varphi)=-\int_K u\varphi d\mu, \quad\forall \varphi\in dom\mathcal{E}.\]
\end{definition}

Although, we will focus on Besov spaces (and Sobolev spaces) with Neumann boundary condition in this paper, it is convenient to consider $\Delta$ instead of $\Delta_N$ in the proof, to enlarge the domain a little bit. 

\begin{definition}
Define $\mathcal{H}_0=\{h\in dom\mathcal{E}:\Delta h=0\}$, and call $h\in \mathcal{H}_0$ a \em{harmonic function}.
\end{definition}

In fact, $\mathcal{H}_0$ is a finite dimensional space, and each $h\in \mathcal{H}_0$ is uniquely determined by its boundary value on $V_0$. 
In particular, we can see that $\mathcal{H}_0$ is always in the $L^p$ domain of $\Delta$ for any $1<p<\infty$.

\subsection{Besov spaces on $K$}
In this paper, we consider the (heat) Besov spaces $B^{p,q}_\sigma(K)$ with the Neumann boundary condition. Recall that $P_t=e^{\Delta_Nt},t>0$ is a heat operator associated with $\Delta_N$, and the Bessel potential can be defined as  $(1-\Delta_N)^{-\sigma/2}=\Gamma(\sigma/2)^{-1}\int_0^\infty t^{\sigma/2-1}e^{-t}P_tdt$. We define potential spaces on $K$ as follows, following \cite{HZ} and \cite{s1}.

\begin{definition}
(a). For $1<p<\infty$, $\sigma\geq 0$, define the {\em Sobolev space} $$H^p_\sigma(K)=(1-\Delta_N)^{-\sigma/2}L^p(K),$$ with norm $\|f\|_{H^p_\sigma(K)}=\big\|(1-\Delta_N)^{\sigma/2}f\big\|_{L^p(K)}$.

(b). For $1<p<\infty$, $1\leq q\leq\infty$ and $\sigma>0$, define the {\em heat Besov space} 	 
\[B^{p,q}_\sigma(K)=\Big\{f\in L^p(K):\Big(\int_{0}^\infty \big(t^{-\sigma/2}\big\|(t\Delta_N)^kP_tf\big\|_{L^p(K)}\big)^{q}dt/t\Big)^{1/q}<\infty\Big\},\]
with $k\in\mathbb{N}\cap({\sigma}/{2},\infty)$, and norm $\|f\|_{B^{p,q}_\sigma(K)}=\|f\|_{L^p(K)}+\big(\int_{0}^\infty (t^{-\sigma/2}\|(t\Delta_N)^kP_tf\|_{L^p(K)})^{q}dt/t\big)^{1/q}$. We take the usual modification when $q=\infty$.
\end{definition}

Note that the above definition is independent of $k$, since different choices of $k$ will provide equivalent norms, see \cite{GL} for example. The heat Besov spaces are related with Sobolev spaces by real interpolation. See book \cite{sectorial} for a proof, noticing that $\Delta_N$ is a sectorial operator. See also books \cite{interpolation,tribel} for the real interpolation methods.

\begin{lemma}\label{lemma26}
Let $\sigma_1>0$, $1<p<\infty$ and $1\leq q\leq\infty$. For $0<\theta<1$ and $\sigma_\theta=\theta\sigma_1$, we have 
\[\big(L^p(K),H^p_\sigma(K)\big)_{\theta,q}=B^{p,q}_{\sigma_\theta}(K).\]
\end{lemma}

In application, we will set $\sigma_1=2$ in the above lemma, where $H^p_2(K)=dom_{L^p(K)}\Delta_N:=\big\{f\in L^p(K):\Delta_Nf\in L^p(K)\big\}$. See Section 5 and 6 for details. \vspace{0.2cm}

Another class of function spaces that will be studied is the Lipschitz-Besov spaces, whose definition does not rely on the Laplacian.
\begin{definition}
Let $1\leq p<\infty$, $t>0$ and $f$ be a measurable function on $K$, we define 
\[I_p(f,t)=\big(\int_Kt^{-d_H}\int_{B_t(x)}|f(x)-f(y)|^pd\mu(y)d\mu(x)\big)^{1/p}.\]
In addition, we define $I_\infty(f,t)=\sup\big\{|f(x)-f(y)|: x,y\in K, R(x,y)< t\big\}$.
\end{definition}

The \textit{Lipschitz-Besov spaces}, denote by $\Lambda^{p,q}_\sigma(K)$, are defined as follows.

\begin{definition}\label{def24}
	For $\sigma>0$ and $1\leq p,q\leq \infty$, we define 
	\[\Lambda^{p,q}_\sigma(K)=\big\{f\in L^p(K): t^{-\sigma d_W/2} I_p(f,t)\in L^q_*(0,1]\big\},\]
	with norm 
	\[\|f\|_{\Lambda^{p,q}_\sigma(K)}:=\|f\|_{L^p(K)}+\big\| t^{-\sigma d_W/2} I_p(f,t)\big\|_{L^q_*(0,1]},\]
	where $\|f\|_{L^q_*(0,1]}=\big(\int_0^1|f(t)|^q\frac{dt}{t}\big)^{1/q}$ and we take the usual modification when $q=\infty$.
\end{definition}

\noindent\textbf{Remark.} Since $K$ is bounded, we can replace the integral of $t$ over $(0,1]$ with $(0,\infty)$ in the above definition.

\section{A critical curve}
In this section, we  introduce a \textit{critical curve} $\mathscr{C}$ in the $(\frac{1}{p},\sigma)$-parameter plane as follows. 

\begin{definition}\label{def31}
For $1\leq p\leq \infty$, we define $\mathscr{C}(p)=\sup\big\{\sigma>0:\mathcal{H}_0\subset \Lambda^{p,\infty}_\sigma(K)\big\}$.
\end{definition}  

The critical curve $\mathscr{C}$ is sharp, since for any $\sigma>\mathscr{C}(p)$, we have locally harmonic functions contained in $B^{p,q}_\sigma(K)\setminus \Lambda^{p,q}_{\sigma}(K)$.

\subsection{Two regions}
In this part, we provide some qualitative behavior of the critical curve $\mathscr C$. We begin with the following easy observation.  

\begin{proposition}\label{prop33}
(a). The critical curve $\mathscr C$ is concave and increasing with respect to the parameter $\frac{1}{p}$. In addition, $\mathscr C(\infty)=\frac{2}{d_W}$ and $\mathscr C(2)=1$. 

(b). For $1\leq p\leq 2$, we have $1\leq \mathscr{C}(p)\leq 1+(\frac{2}{p}-1)(d_S-1)$.

(c). For $2\leq p\leq \infty$, we have $1+(\frac{2}{p}-1)(d_S-1)\leq \mathscr{C}(p)\leq 1\wedge (\frac{2}{d_W}+\frac{d_S}{p})$.

\noindent See Figure \ref{criticalcurve} for an illustration.
\end{proposition}
\begin{figure}[h]
	\includegraphics[width=5cm]{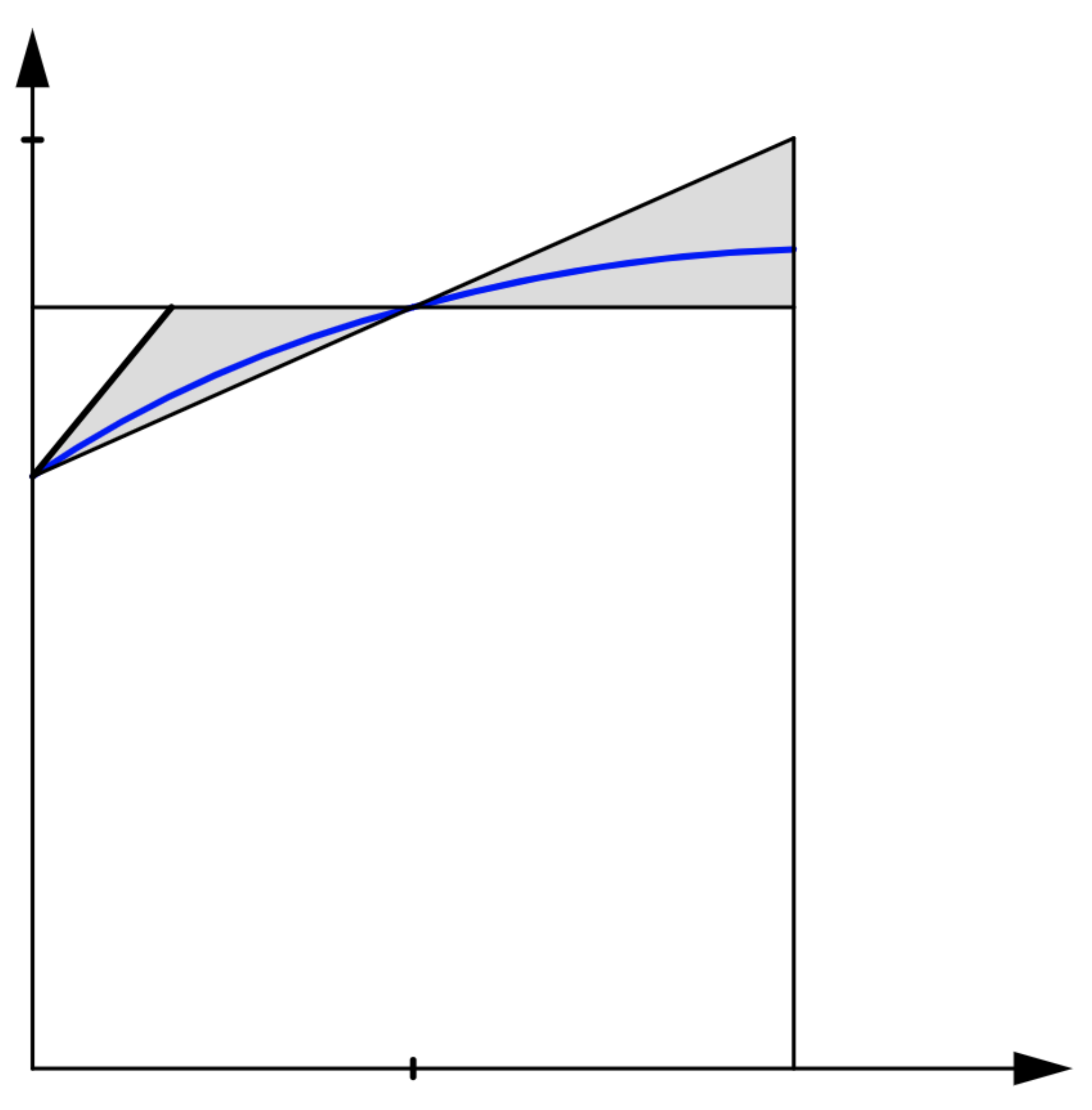}
	\begin{picture}(0,0)
	\put(-150,135){$\sigma$}
	\put(-151,101){$1$}
	\put(-155,122){$d_S$}
	\put(-157,78){$\frac{2}{d_W}$}
	\put(-26,14){$\frac 1p$}\put(-40,108){$\mathscr C$}
	\put(-45,-6){$1$}\put(-95,-8){$\frac{1}{2}$}
	\end{picture}
	\caption{The critical curve $\mathscr C$ in the $(\frac{1}{p},\sigma)$-parameter plane.}
	\label{criticalcurve}
\end{figure}
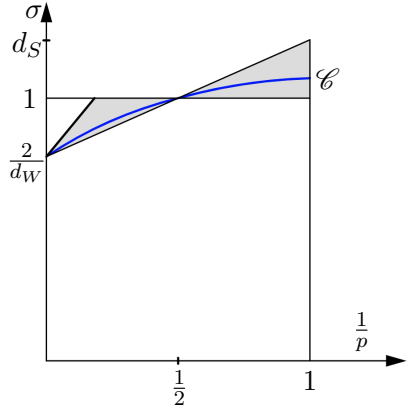
\begin{proof}
Recall that $d_W=1+d_H$, $d_S=\frac{2d_H}{d_W}$, and note that $d_S-1=1-\frac{2}{d_W}$. 

(a). The observation that $\mathscr{C}(\infty)=\frac{2}{d_W}$ follows from the fact that $0<\sup_{x\neq y}\frac{|h(x)-h(y)|}{R(x,y)}<\infty$ for any non-constant harmonic function $h$, see \cite{T1}.  For $p=2$, it is well known that $\Lambda^{2,\infty}_1(K)=dom\mathcal{E}$ and $\Lambda^{2,\infty}_\sigma(K)=constants$ provided $\sigma>1$ (\cite{GHL}), which gives $\mathscr{C}(2)=1$. 
	
Next, let $1\leq p_1<p_2\leq\infty$, $\sigma_1<\mathscr{C}(p_1)$ and $\sigma_2<\mathscr{C}(p_2)$. Also, let $\frac{1}{p}=\frac{1}{2p_1}+\frac{1}{2p_2}$, and $\sigma=\frac{\sigma_1+\sigma_2}{2}$. Then for any $0<t\leq 1$ and $h\in \mathcal{H}_0$, it holds that 
\[\begin{aligned}&\big(\int_Kt^{-d_H}\int_{B_t(x)}|h(x)-h(y)|^pd\mu(y)d\mu(x)\big)^{1/p}\\&\leq \big(\int_Kt^{-d_H}\int_{B_t(x)}\sqrt{|h(x)-h(y)|}^{2p_1}d\mu(y)d\mu(x)\big)^{1/2p_1}\cdot \big(\int_Kt^{-d_H}\int_{B_t(x)}\sqrt{|h(x)-h(y)|}^{2p_2}d\mu(y)d\mu(x)\big)^{1/2p_2}\end{aligned}\]
and thus
\[t^{-\sigma d_W/2}I_p(h,t)\leq t^{-\sigma d_W/2}\sqrt{I_{p_1}(h,t)}\sqrt{I_{p_2}(h,t)}\leq \sqrt{\|h\|_{\Lambda^{p_1,\infty}_{\sigma_1}(K)}}\sqrt{\|h\|_{\Lambda^{p_2,\infty}_{\sigma_2}(K)}}.\]
 This implies $\mathcal{H}_0\subset \Lambda^{p,\infty}_\sigma(K)$. Thus, we conclude $\mathscr{C}(p)\geq \frac{1}{2}\big(\mathscr{C}(p_1)+\mathscr{C}(p_2)\big)$. So $\mathscr{C}$ is concave.
 
 Lastly, there is a constant $C>0$ such that $\mu(B_t(x))\leq C t^{d_H}$ for any $x\in K$ and $t\in(0,1]$. Thus, for $1\leq p_1\leq p_2<\infty$ and $0<t\leq 1$, it is easy to see
 \[I_{p_1}(h,t)\leq CI_{p_2}(h,t)\]
 by using the {H\"{o}lder} inequality. This implies that $\mathscr{C}$ is increasing with respect to $\frac1p$. 

 \vspace{0.2cm}

(b).  Part (b) is a consequence of part (a) and the fact that $\mathscr{C}(\infty)=\frac{2}{d_W}$ and $\mathscr{C}(2)=1$. \vspace{0.2cm}

(c). Now by part (a), we can conclude that $1+(\frac{2}{p}-1)(d_S-1)\leq \mathscr{C}(p)\leq 1$. It remains to prove $\mathscr{C}(p)\leq \frac{2}{d_W}+\frac{d_S}{p}$. We choose a non-constant harmonic function $h$ such that $h\circ F_1=r_1h$. For any $n\geq 0$, we  see that 
\[\begin{aligned}
I_p(h,r_1^n)&=\big(\int_K\mu_1^{-n}\int_{B_{r_1^n}(x)}|h(x)-h(y)|^pd\mu(y)d\mu(x)\big)^{1/p}\\
&\geq \big(\int_{F_1^nK}\mu_1^{-n}\int_{B_{r_1^n}(x)}|h(x)-h(y)|^pd\mu(y)d\mu(x)\big)^{1/p}\\
&\geq r_1^n\mu_1^{n/p}\big(\int_{K}\int_{K}|h(x)-h(y)|^pd\mu(y)d\mu(x)\big)^{1/p}.
\end{aligned}\]
This implies that $r_1^{-\mathscr{C}(p)d_W/2}r_1^{1+d_H/p}\leq 1$, and thus $\mathscr{C}(p)\leq \frac{2}{d_W}+\frac{d_S}{p}$.
\end{proof}

\noindent\textbf{Remark.} (a). Proposition \ref{prop33} (c) implies that when $d_H>1$, we can not expect that $B^{p,q}_\sigma(K)=\Lambda^{p,q}_\sigma(K)$ holds for any $1<p<\infty$ and $0<\sigma<1$. 

\noindent(b). For $1<p<2$, it is possible that $B^{p,q}_\sigma(K)=\Lambda^{p,q}_\sigma(K)$ for some $\sigma>1$. See the next subsection for examples with $\mathscr{C}(1)>1$.\vspace{0.2cm}

\begin{figure}[h]
	\includegraphics[width=5cm]{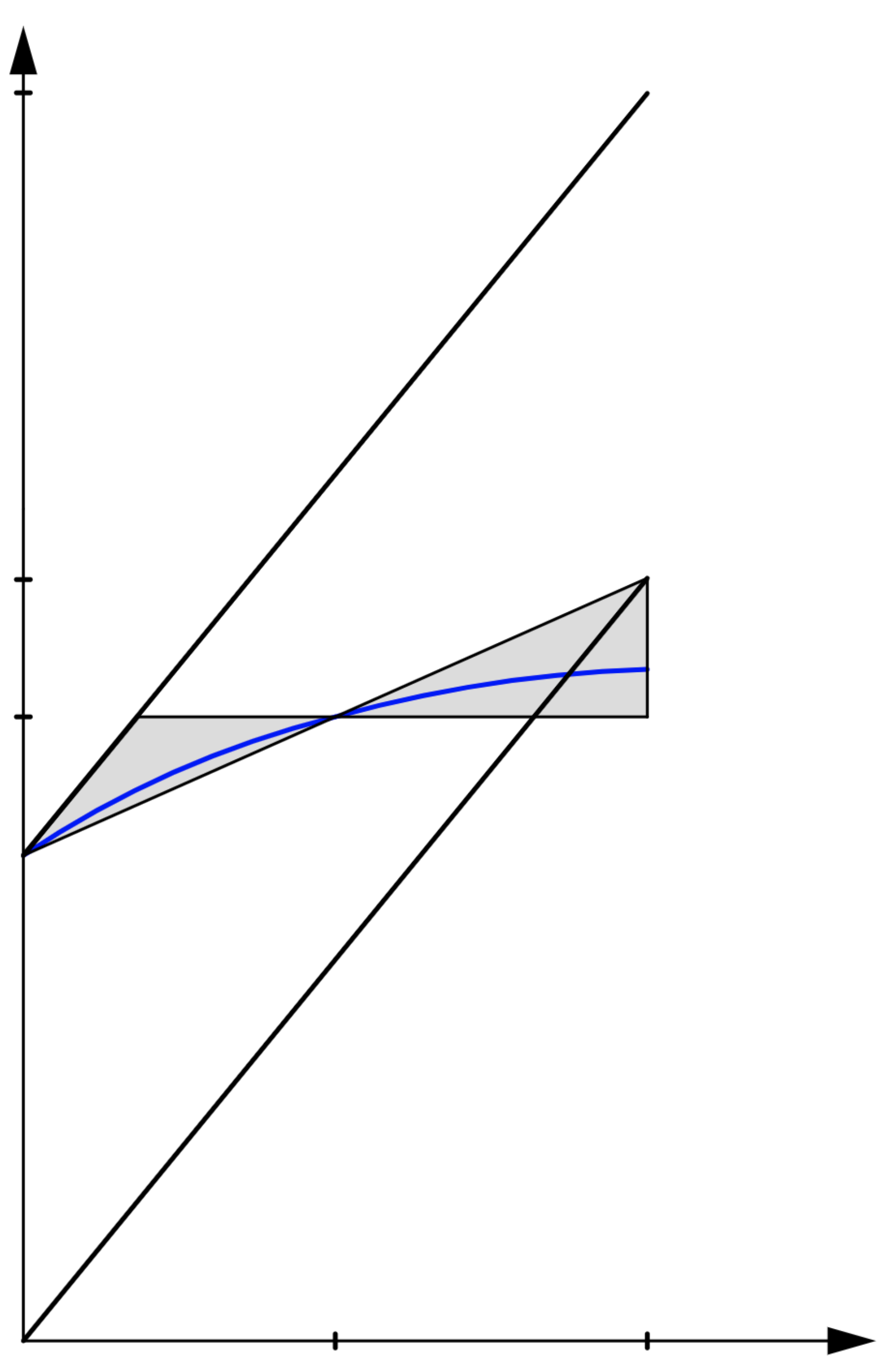}
	\begin{picture}(0,0)
	\put(-150,212){$\sigma$}
	\put(-151,101){$1$}	\put(-151,201){$2$}
	\put(-155,122){$d_S$}
	\put(-157,78){$\frac{2}{d_W}$}
	\put(-26,14){$\frac 1p$}\put(-40,108){$\mathscr C$}\put(-84,70){$\mathscr L_1$}\put(-84,148){$\mathscr L_2$}
	\put(-45,-6){$1$}\put(-95,-8){$\frac{1}{2}$}
	\end{picture}
	\caption{The critical curves $\mathscr L_1$, $\mathscr L_2$ and $\mathscr C$.}
	\label{criticallines}
\end{figure}

There are two more critical lines $\mathscr{L}_1$, $\mathscr{L}_2$ in the $(\frac{1}{p},\sigma)$-parameter plane, that are of interest, with 
$$\mathscr{L}_1(p)=\frac{d_S}{p},\text{ and }\mathscr{L}_2(p)=2-\frac{d_S}{p'},$$
where $p'=\frac{p}{p-1}$.
See Figure \ref{criticallines} for an illustration for the positions of $\mathscr{C},\mathscr{L}_1$ and $\mathscr{L}_2$. In particular, as illustrated in \cite{GHL,HZ,s1}, the Sobolev spaces $H^p_\sigma(K)$ and the heat Besov spaces $B^{p,q}_\sigma(K)$ are embedded in $C(K)$ when the parameter point  $(\frac{1}{p},\sigma)$ is above $\mathscr{L}_1$, and these function spaces with or without Neumann condition coincide if $(\frac{1}{p},\sigma)$ is below $\mathscr{L}_2$ (\cite{cq1,cq2,cq3}), which clearly covers the parameter region below $\mathscr{C}$ by Proposition \ref{prop33}.

In this paper, we are most interested in the region $\sigma<\mathscr{C}(p)$, and we can see that $\mathscr{C}$ and $\mathscr{L}_1$ intersect at some point with $1\leq p\leq d_S$ by Proposition \ref{prop33}. In particular, we divide the region below $\mathscr{C}$ into two parts, see Figure \ref{twoareas} for an illustration. \vspace{0.2cm}

\noindent \textbf{Region 1}. $\mathscr{A}_1:=\big\{(\frac{1}{p},\sigma): 1<p<\infty\text{ and } \mathscr{L}_1(p)<\sigma<\mathscr{C}(p)\big\}$;

\noindent \textbf{Region 2}.  $\mathscr{A}_2:=\big\{(\frac 1 p,\sigma): 1<p<\infty\text{ and }0< \sigma<\mathscr{L}_1(p)\wedge\mathscr{C}(p)\big\}$.

\vspace{0.2cm}

\begin{figure}[h]
	\includegraphics[width=5cm]{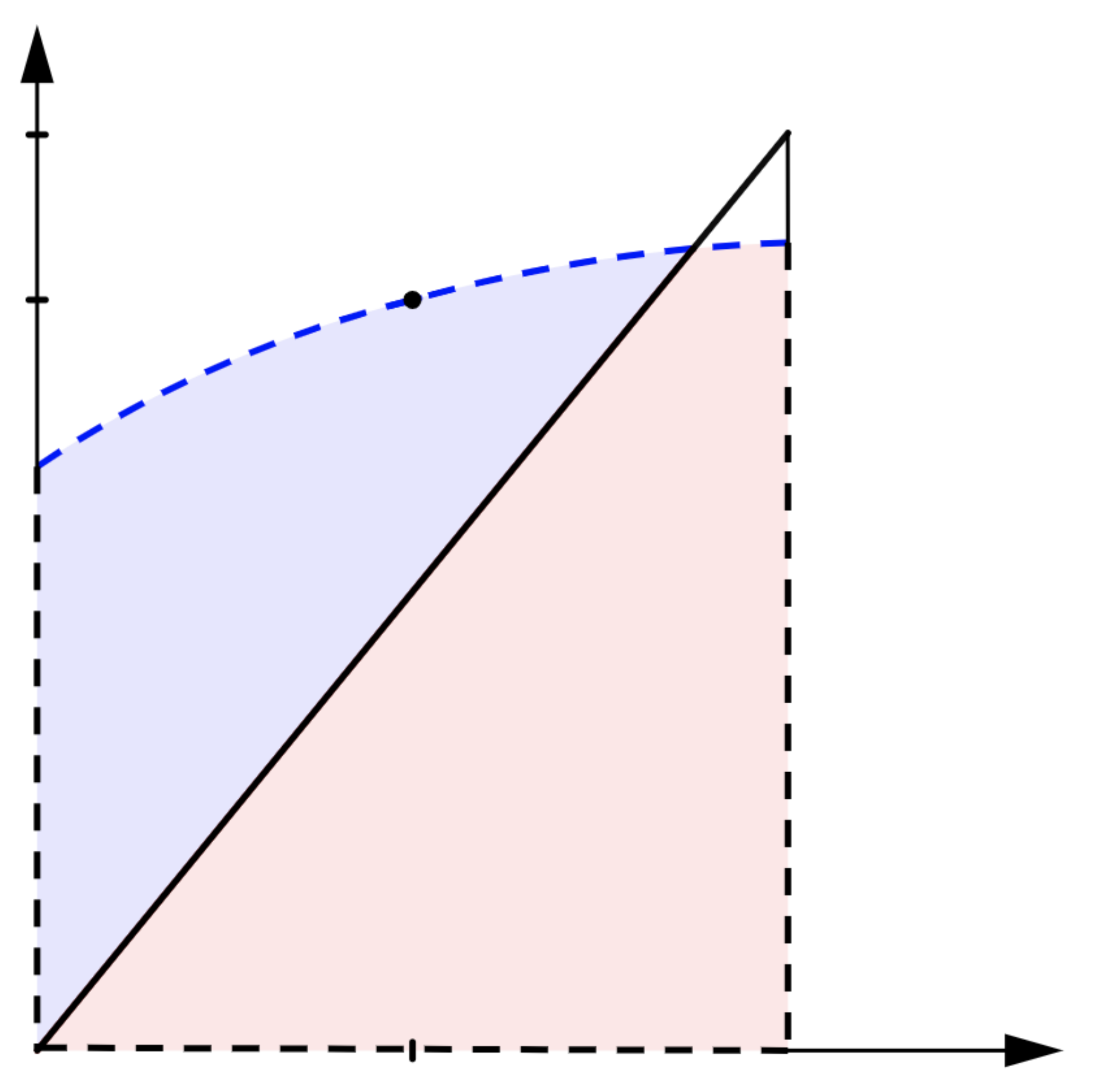}
	\begin{picture}(0,0)
	\put(-152,135){$\sigma$}
	\put(-151,101){$1$}
	\put(-157,122){$d_S$}
	\put(-157,78){$\frac{2}{d_W}$}
	\put(-26,14){$\frac 1p$}\put(-80,40){$\mathscr A_2$}\put(-120,70){$\mathscr A_1$}
	\put(-45,-6){$1$}\put(-95,-8){$\frac{1}{2}$}
	\end{picture}
	\caption{The regions $\mathscr A_1$ and $\mathscr A_2$.}
	\label{twoareas}
\end{figure}

 We will apply different methods when considering these two regions, for the proof of $B^{p,q}_\sigma(K)\subset \Lambda^{p,q}_\sigma(K)$. The border between the two regions can be dealt with by using real interpolation. 

The reason that we need to divide the region $\sigma<\mathscr{C}(p)$ in this manner is due to the existence of the region $\mathscr{C}(p)<\sigma<\mathscr{L}_1(p)$ when $\mathscr{C}(1)<d_S$. For example, this happens for the Sierpinski gasket, see the next subsection.

\subsection{Examples}
In this subsection, we look at some typical p.c.f. self-similar sets, and describe their critical curves $\mathscr C$ or provide some rough estimates.\vspace{0.2cm}

\noindent\textbf{Example 1.} The \textit{unit interval} $I=[0,1]$, generated by $F_1(x)=\frac x 2, F_2(x)=\frac x 2+\frac 1 2$, is a simplest example of p.c.f. self-similar sets. We equip $I$ with the standard Laplacian, then it has walk dimension $d_W=2$ and spectral dimension $d_S=1$. So the critical curve is simply a horizontal line segment, $\mathscr{C}(p)\equiv 1$.\vspace{0.2cm}

\noindent\textbf{Example 2.} A more interesting example is the \textit{Vicsek set} $\mathcal{V}$. Let $\{q_i\}_{i=1}^4$ be the four vertices of a square in $\mathbb{R}^2$, and let $q_5$ be the center of the square. Define an i.f.s. $\{F_i\}_{i=1}^5$ by \[F_i(x)=\frac{1}{3}(x-q_i)+q_i,\text{ for }1\leq i\leq 5.\]
The Vicsek set $\mathcal{V}$ is then the unique compact set in the square such that $\mathcal{V}=\bigcup_{i=1}^5 F_i\mathcal{V}$, see Figure \ref{figexample21}.  

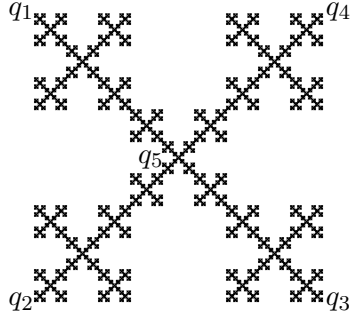
\begin{figure}[htp]
	\includegraphics[width=4cm]{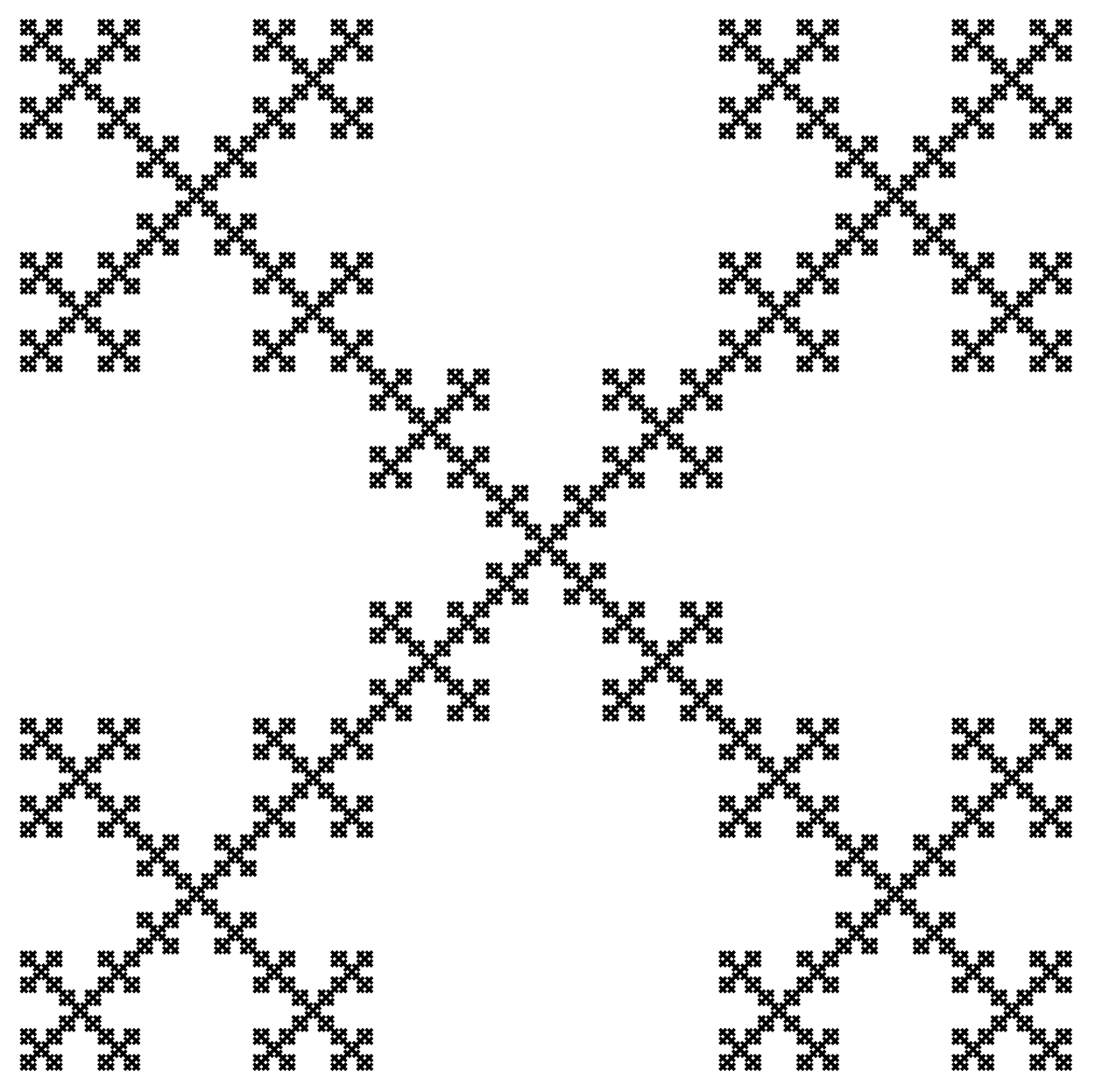}
	\begin{picture}(0,0)
	\put(-125,111){$q_1$}
	\put(-125,1){$q_2$}
	\put(-5,1){$q_3$}
	\put(-5,111){$q_4$}
	\put(-76,55){$q_5$}
	\end{picture}
	\caption{The Vicsek set $\mathcal{V}$.}\label{figexample21}
\end{figure}

We equip $\mathcal V$ with the fully symmetric measure $\mu$ and energy form $(\mathcal E, dom\mathcal E)$. In particular, $\mu$ is chosen to be the normalized Hausdorff measure on $\mathcal{V}$. As for $(\mathcal E, dom\mathcal E)$, recall that it could be defined first on discrete graphs on $V_m$'s then passing to the limit. Note that $V_m=\bigcup_{w\in W_m}F_wV_{0}$, where $V_0=\{q_1,q_2,q_3,q_4\}$ is the boundary of $\mathcal V$. For convenience of the later calculation, we instead to use an equivalent definition of $(\mathcal E, dom\mathcal E)$ by involving the point $q_5$ in the graph energy forms, i.e., letting $\tilde{V}_0=\{q_i\}_{i=1}^5$ and $\tilde{V}_m=\bigcup_{w\in W_m}F_w\tilde{V}_0$, and defining the energy form on $\tilde{V}_0$ to be
\[\tilde{\mathcal{E}}_0(f,g)=\sum_{i=1}^4\big(f(q_i)-f(q_5)\big)\big(g(q_i)-g(q_5)\big),\]
and iteratively $\tilde{\mathcal{E}}_m(f,g)=3\sum_{i=1}^5 \tilde{\mathcal{E}}_{m-1}(f\circ F_i, g\circ F_i)$ on $\tilde{V}_m$, which still approximate  $(\mathcal{E}, dom\mathcal{E})$ on $\mathcal V$. In particular, we have $r=\frac 13$, and in addition,
\[d_H=\frac{\log5}{\log3},\quad d_W=1+d_H=\frac{\log15}{\log3},\quad d_S=\frac{2d_H}{d_W}=\frac{2\log{5}}{\log15}.\]
We will show that $\mathscr C(p)$ is a line segment with slope $2(d_S-1)$, i.e. 
\begin{equation}\label{eqn31}
\mathscr{C}(p)=1+(\frac{2}{p}-1)(d_S-1)=\frac{2\log3}{\log{15}}+\frac{2}{p}\cdot\frac{\log5-\log3}{\log15}.
\end{equation}

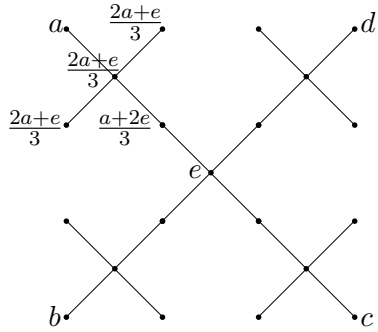
\begin{figure}[htp]
	\includegraphics[width=4cm]{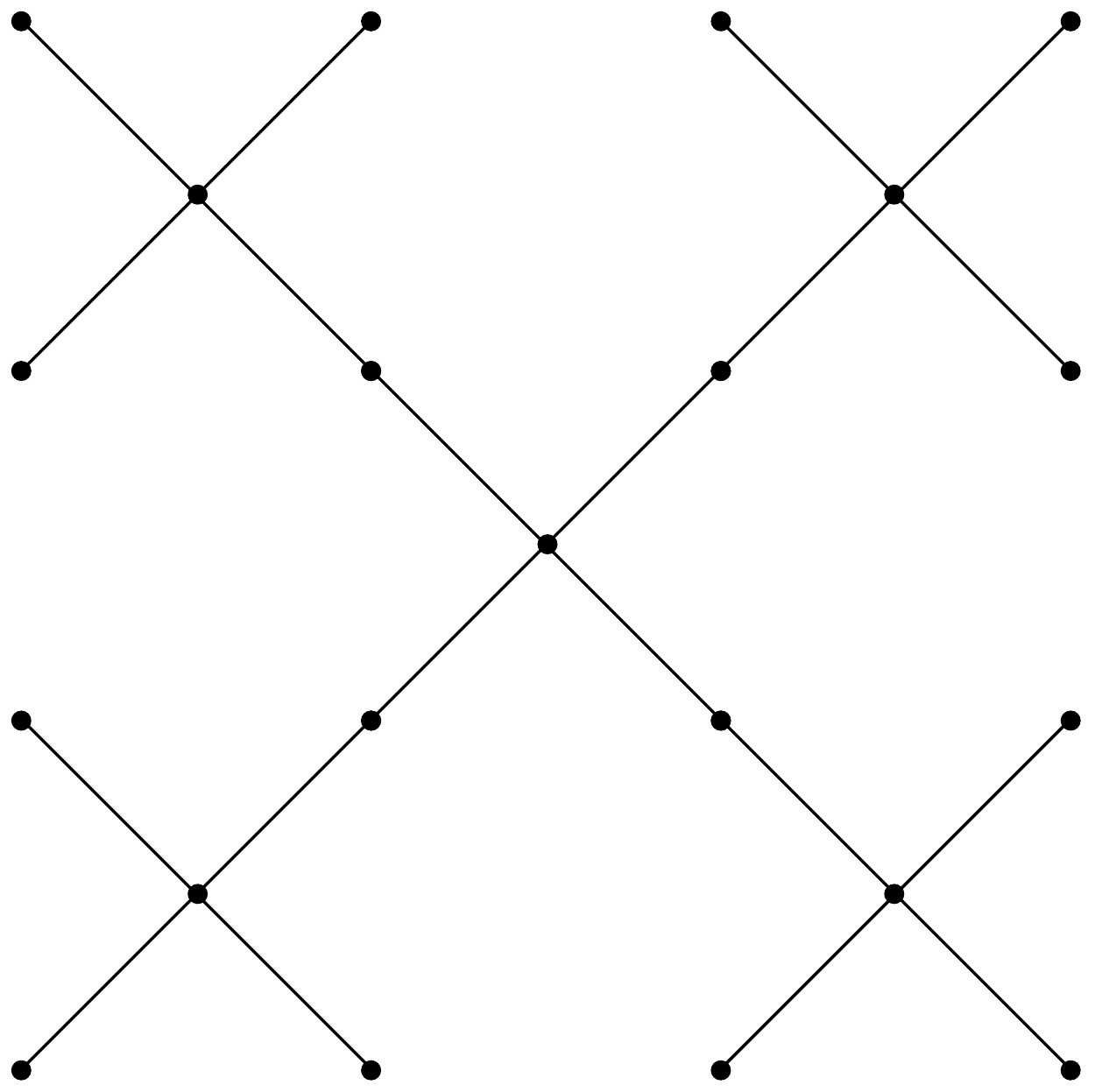}
	\begin{picture}(0,0)
	\put(-122,110){$a$}
	\put(-122,-1){$b$}
	\put(-4,-1){$c$}
	\put(-4,110){$d$}
	\put(-69,55){$e$}
	
	\put(-138,71){$\frac{2a+e}{3}$}
	\put(-116,93){$\frac{2a+e}{3}$}
	\put(-100,111){$\frac{2a+e}{3}$}
	\put(-104,71){$\frac{a+2e}{3}$}
	\end{picture}
	\caption{A harmonic function $h$ on $\mathcal{V}$ with boundary value $h(q_1)=a$, $h(q_2)=b$, $h(q_3)=c$, $h(q_4)=d$, and $e=h(q_5)=(a+b+c+d)/4$.}\label{figexample22}
\end{figure}

In particular, for $h\in \mathcal{H}_0$ and $t\in (0,1]$, we are interested in the estimate of $I_1(h,t)$. We denote by $\sum_{x\sim_m y}\big|h(x)-h(y)\big|$ the sum of absolute differences of $h$ over edges of level $m$, where $x\sim_m y$ means that there exist a word $w\in W_m$ and an $1\leq i\leq 4$ such that $x=F_wq_i$ and $y=F_wq_5$. Since $\mathcal{H}_0$ is of finite dimension, it is not hard to check that 
\[\sum_{x\sim_m y}\big|h(x)-h(y)\big|\asymp 5^{m}I_1(h,3^{-m})=3^{md_H}I_1(h,3^{-m}).\]
On the other hand, due to the harmonic extension algorithm as shown in Figure \ref{figexample22}, we immediately have 
\[\sum_{x\sim_m y}\big|h(x)-h(y)\big|=\sum_{i=1}^4 |f(q_i)-f(q_5)|,\quad \forall m\geq 0.\]
So $\sup_{m\geq 0}3^{m d_Sd_W/2}I_1(h,3^{-m})\lesssim \|h\|_\infty$, which means $h\in \Lambda^{1,\infty}_{d_S}(\mathcal{V})$. Thus $\mathscr{C}(1)=d_S$ by applying Proposition \ref{prop33} (b). This determines the formula of $\mathscr{C}(p)$ in (\ref{eqn31}), using Proposition \ref{prop33} (a).

The above description of $\mathscr C$ is also valid for a general $(2k+1)$-Vicsek set with $k\geq 1$, which is generated by an i.f.s. of $4k+1$ contractions, such that each of the two cross directions of the fractal consists of $2k+1$ sub-cells. We omit the details.\vspace{0.2cm}

\noindent\textbf{Example 3.} The next example is the \textit{Sierpinski gasket} $\SG$. Let $\{q_i\}_{i=0}^3$ be the three vertices of an equilateral triangle in $\mathbb{R}^2$, and define an i.f.s. $\{F_i\}_{i=1}^3$ by 
\[F_i(x)=\frac{1}{2}(x-q_i)+q_i,\text{ for }1\leq i\leq 3.\]
The Sierpinski gasket $\mathcal{SG}$ is the unique compact set in $\mathbb{R}^2$ such that $\SG=\bigcup_{i=1}^3 F_i\SG$, see Figure \ref{figexample31}. 

\begin{figure}[htp]
	\includegraphics[width=4cm]{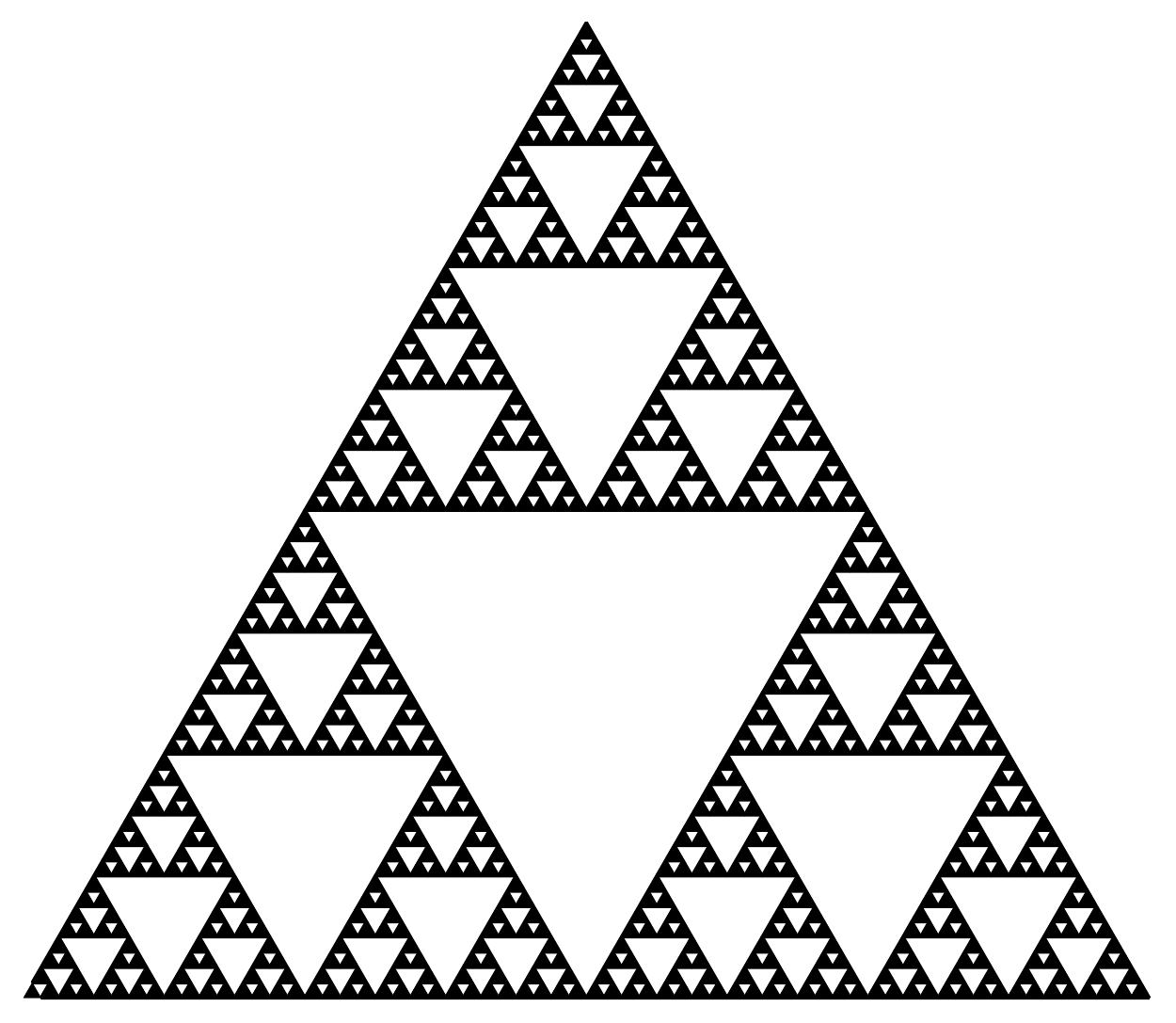}
	\begin{picture}(0,0)
	\put(-66,100){$q_1$}
	\put(-125,1){$q_2$}
	\put(-5,1){$q_3$}
	\end{picture}
	\caption{The Sierpinski gasket $\SG$.}\label{figexample31}
\end{figure}
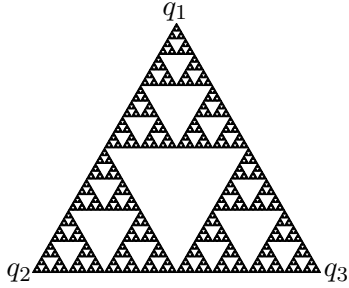

On $\SG$, we take the normalized Hausdroff measure $\mu$ and the standard energy form $(\mathcal{E}, dom\mathcal E)$ satisfying 
\[\mathcal{E}(f,g)=\frac{5}{3}\sum_{i=1}^3\mathcal{E}(f\circ F_i, g\circ F_i), \quad \forall f,g\in dom\mathcal{E}.\]
In particular, we have $r=\frac 3 5$, and in addition,
\[d_H=\frac{\log3}{\log5-\log3},\quad d_W=1+d_H=\frac{\log5}{\log5-\log3},\quad d_S=\frac{2d_H}{d_W}=\frac{2\log3}{\log5}{\approx1.36521}.\]

It seems hard to get the exact formula of $\mathscr{C}(p)$. However, we will figure out that $\mathscr{C}(p)$ is indeed a ``curve'' by observing that $1<\mathscr C(1)<d_S$, and then using Proposition \ref{prop33}. In fact, this can be verified by estimating the maximal exponential growth ratio  of 
$\sum_{x\sim_m y}\big|h(x)-h(y)\big|$ as $m\rightarrow\infty$, which should be  $ r^{\mathscr C(1)d_W/2-d_H}$ for harmonic functions $h$ on $\mathcal {SG}$. Since any harmonic function $h$ is a combination of $h_1, h_2, h_3$ with $h_i(q_j)=\delta_{i,j}$, by calculating $\sum_{x\sim_m y}\big|h_1(x)-h_1(y)\big|$ with $m=3$, we  see that
\[1.02<\mathscr{C}(1)<1.14.\]

\section{Discrete characterizations of $\Lambda^{p,q}_\sigma(K)$}
In this section, we will provide some discrete characterizations of the Lipschitz-Besov spaces $\Lambda^{p,q}_\sigma(K)$. These characterizations will provide great convenience in proving Theorem \ref{thm11}. In particular, they  heavily rely on the nested structure of $K$.

\begin{definition}\label{def41}
(a). For $m\geq 0$, define $\Lambda_m=\{w\in W_*:r_w\leq r^m<r_{w^*}\}$ with $r=\min_{1\leq i\leq N}r_i$. In particular, we denote $\Lambda_0={\{\emptyset\}}$. 

(b). Define $V_{\Lambda_m}=\bigcup_{w\in \Lambda_m} F_wV_0$ for $m\geq 0$, and denote 
$$\mathring{V}_{\Lambda_m}=
\begin{cases}
V_0,&\text{ if }m=0,\\
V_{\Lambda_m}\setminus V_{\Lambda_{m-1}},&\text{ if }m\geq 1.
\end{cases}$$ 
\end{definition}

In the rest of this section, we will consider two kinds of discrete characterizations of $\Lambda^{p,q}_\sigma(K)$, basing on the cell graphs approximation and vertex graphs approximation of $K$ respectively.

\subsection{A Haar series expansion}
We begin with a Haar series expansion of a function. We classify Haar functions on $K$ into different levels based on the partition $\Lambda_m$.

\begin{definition}\label{def42}
(a). For each $f\in L^1(K)$, we define $E_w(f)=\frac{1}{\mu_w}\int_{F_wK}fd\mu$, and write $$E[f|\Lambda_m]=\sum_{w\in \Lambda_m}E_w(f)1_{F_wK},\quad m\geq 0,$$  
which can be understood as the conditional expectation of $f$ with respect to the sigma algebra generated by the collection  $\{F_wK: w\in \Lambda_m\}$.
In addition, we write \[\tilde{E}[f|\Lambda_m]=\begin{cases}
E[f|\Lambda_0], &\text{ if }m=0,\\
E[f|\Lambda_m]-E[f|\Lambda_{m-1}], &\text{ if }m\geq 1.
\end{cases}\]

\noindent(b). Define $\tilde{J}_m=\big\{\tilde{E}[f|\Lambda_m]:f\in L^1(K)\big\}$, and call $\tilde{J}_m$ the space of level-$m$ Haar functions. 
\end{definition}

It is easy to see that for $m\geq 1$, $\tilde{J}_m$ consists of functions $u$ which are piecewise constant on $\{F_wK: w\in \Lambda_m\}$, and satisfy $E[u|\Lambda_{m-1}]=0$. We have the following estimates.

\begin{lemma}\label{lemma43}
Let $f\in L^p(K)$ with $1<p<\infty$ and $u\in \tilde{J}_m$ with $m\geq 0$. Then

(a). $\big\|\tilde{E}[f|\Lambda_m]\big\|_{L^p(K)}\leq C I_p(f,r^{m-1})$ for any $m\geq 1$.

(b). $I_p(f,t)\leq C\|f\|_{L^p(K)}$ for any $0<t\leq 1$.
	
(c).  $I_p(u,r^n)\leq Cr^{(n-m)d_H/p}\|u\|_{L^p(K)}$ for any $n\geq m$.

\noindent The constant $C$ can be chosen to be  independent of $f,u$ and $p$.
\end{lemma}

\begin{proof} (a). For each point $x\in K\setminus V_{\Lambda_m}$, we define $Z_{\Lambda_m}(x)=F_wK$ with $w\in \Lambda_{m}$ such that $x\in F_wK$. Clearly, we have 
$Z_{\Lambda_m}(x)\subset B_{r^{m}}(x)$ since the diameter of each cell $F_wK,w\in \Lambda_m$ is at most $r^{m}$. For $m\geq 1$, we have 
\[\begin{aligned}
\big\|\tilde{E}[f|\Lambda_m]\big\|_{L^p(K)}&=\big(\int_{K}\big|E[f|\Lambda_m](x)-E[f|\Lambda_{m-1}](x)\big|^pd\mu(x)\big)^{1/p}\\
&\leq \big(\int_{K}\big|f(x)-{E}[f|\Lambda_{m-1}](x)\big|^pd\mu(x)\big)^{1/p}\\
&\leq \big(\int_K\big(\mu_{Z_{{\Lambda_{m-1}(x)}}}^{-1}\int_{Z_{\Lambda_{m-1}(x)}}|f(x)-f(y)|^pd\mu(y)\big)d\mu(x)\big)^{1/p}\\
&\leq \big(\int_Kr^{-md_H}\int_{B_{r^{m-1}}(x)}|f(x)-f(y)|^pd\mu(y)d\mu(x)\big)^{1/p}\leq CI_p(f,r^{m-1}),
\end{aligned}\]
where we ignore the finitely many points in $V_{\Lambda_m}$ in the above estimate. \vspace{0.2cm}

(b) is obvious, and $C$ only depends on the estimate $\mu(B_t(x))\lesssim t^{d_H}$.\vspace{0.2cm}

(c). First, we have the estimate that
\begin{equation}\label{eqn41}
\begin{aligned}
\|u\|_{L^p(K)}&=\big(\sum_{w\in\Lambda_m}\mu_w|E_w(u)|^p\big)^{1/p}\geq \big(r^{(m+1)d_H}\sum_{w\in \Lambda_m}|E_w(u)|^p\big)^{1/p}\\
&\gtrsim r^{md_H/p}\big(\sum_{w\sim w'\text{ in }\Lambda_m}|E_w(u)-E_{w'}(u)|^p\big)^{1/p},
\end{aligned}
\end{equation}
where we write $w\sim w'$ if $w\neq w'$ and $F_wK\cap F_{w'}K\neq \emptyset$, and use the fact that $\#\{w'\in \Lambda_m:w'\sim w\}\leq \#V_0\#\mathcal{C}$ for any $w\in\Lambda_m$. 

Next, notice that there is $k>0$ such that $R(x,y)>r^n$ for any $m\geq 0$, $n\geq m+k$ and $x\in F_wK, y\in F_{w'}K$ with $F_wK\cap F_{w'}K=\emptyset$, $w,w'\in \Lambda_m$. It suffices to consider $n\geq m+k$, since for $n<m+k$ we have (b). For $n\geq m+k$, we have the estimate 
\[\begin{aligned}
I_p(u,r^n)&=\big(\int_Kr^{-nd_H}\int_{B_{r^{n}}(x)}|u(x)-u(y)|^pd\mu(y)d\mu(x)\big)^{1/p}\\
&\asymp \big(\iint_{R(x,y)<r^n}r^{-nd_H}|u(x)-u(y)|^pd\mu(y)d\mu(x)\big)^{1/p}\\
&=\big(\sum_{w\sim w'\text{ in }\Lambda_m}\iint_{\{x\in F_wK,y\in F_{w'}K:R(x,y)<r^n \}}r^{-nd_H}|u(x)-u(y)|^pd\mu(y)d\mu(x)\big)^{1/p}\\
&\lesssim \big(r^{nd_H}\sum_{w\sim w'\text{ in }\Lambda_m}|E_w(u)-E_{w'}(u)|^p\big)^{1/p}.
\end{aligned}\]
Combining this with the estimate (\ref{eqn41}), we get (c).
\end{proof}

Using Lemma \ref{lemma43}, we can prove a Haar function decomposition of  the spaces $\Lambda^{p,q}_\sigma(K)$ for $0<\sigma<\mathscr{L}_1(p)$. 
\begin{proposition}\label{prop44}
	For $1<p<\infty$, $1\leq q\leq \infty$ and $0<\sigma<\mathscr{L}_1(p)=\frac{d_S}{p}$, we have $f\in \Lambda^{p,q}_\sigma(K)$ if and only if $\big\|r^{-m\sigma d_W/2}\|\tilde{E}[f|\Lambda_m]\|_{L^p(K)}\big\|_{l^q}<\infty$. In addition, 
	$\|f\|_{\Lambda^{p,q}_\sigma(K)}\asymp \big\|r^{-m\sigma  d_W/2}\|\tilde{E}[f|{\Lambda_m}]\|_{L^p(K)}\big\|_{l^q}$.
\end{proposition}
\begin{proof}
	We first observe that
	\begin{equation}\label{eqn42}
	\|f\|_{\Lambda^{p,q}_\sigma(K)}=\|f\|_{L^p(K)}+\big\|t^{-\sigma d_W/2}I_p(f,t)\big\|_{L^q_*(0,1)}\asymp \|f\|_{L^p(K)}+\|r^{-m\sigma d_W/2}I_p(f,r^m)\|_{l^q}.
	\end{equation}
	
	By using Lemma \ref{lemma43} (a), we then easily see that
	\[\big\|r^{-m\sigma d_W/2}\|\tilde{E}[f|{\Lambda_m}]\|_{L^p(K)}\big\|_{l^q}\lesssim \|f\|_{\Lambda^{p,q}_\sigma(K)}.\]
	
	For the other direction, we write $f_m=\tilde{E}[f|\Lambda_m]$ for $m\geq 0$, and assume that 
	\[\big\|r^{-m\sigma d_W/2}\|f_m\|_{L^p(K)}\big\|_{l^q}<\infty.\] 
	Then we have 
	\begin{equation}\label{eqn43}
	\big\|r^{-m\sigma d_W/2}I_p(f,r^m)\big\|_{l^q}\leq \big\|\sum_{n=0}^{m-1} r^{-m\sigma d_W/2}I_p(f_n,r^m)\big\|_{l^q}+\big\|\sum_{n=m}^{\infty} r^{-m\sigma d_W/2}I_p(f_n,r^m)\big\|_{l^q},
	\end{equation}
	with 
	\begin{equation}\label{eqn44}
	\begin{aligned}
	\big\|\sum_{n=0}^{m-1} r^{-m\sigma d_W/2}I_p(f_n,r^m)\big\|_{l^q}&\lesssim \big\|\sum_{n=0}^{m-1} r^{-m\sigma d_W/2+(m-n)d_H/p}\|f_n\|_{L^p(K)}\big\|_{l^q}\\
	&=\big\|\sum_{n=1}^{m} r^{-m\sigma d_W/2+nd_H/p}\|f_{m-n}\|_{L^p(K)}\big\|_{l^q}\\
	&=\big\|\sum_{n=1}^{m} r^{-(m-n)\sigma d_W/2+n(d_H/p-\sigma d_W/2)}\|f_{m-n}\|_{L^p(K)}\big\|_{l^q}\\
	&\leq\big(\sum_{n=1}^\infty r^{n(d_H/p-\sigma d_W/2)}\big)\cdot\big\|r^{-m\sigma d_W/2}\|f_m\|_{L^p(K)}\big\|_{l^q}
	\end{aligned}
	\end{equation}
	by using Lemma \ref{lemma43} (c) and the Minkowski inequality, and 
	\begin{equation}\label{eqn45}
	\begin{aligned}
	\big\|\sum_{n=m}^{\infty} r^{-m\sigma d_W/2}I_p(f_n,r^m)\big\|_{l^q}&\lesssim \big\|\sum_{n=m}^{\infty} r^{-m\sigma d_W/2}\|f_n\|_{L^p(K)}\big\|_{l^q}\\
	&=\big\|\sum_{n=0}^{\infty} r^{-m\sigma d_W/2}\|f_{n+m}\|_{L^p(K)}\big\|_{l^q}\\
	&\leq \big(\sum_{n=0}^\infty r^{n\sigma d_W/2}\big)\cdot\big\| r^{-m\sigma d_W/2}\|f_m\|_{L^p(K)}\big\|_{l^q}
	\end{aligned}
	\end{equation}
	by using Lemma \ref{lemma43} (b) and the Minkowski inequality again.
	Combining equations (\ref{eqn42}), (\ref{eqn43}), (\ref{eqn44}), (\ref{eqn45}), and noticing that $0<\sigma<\frac{d_S}{p}$, we get that
	\[\|f\|_{\Lambda^{p,q}_\sigma(K)}\lesssim \big\|r^{-m\sigma d_W/2}\|f_m\|_{L^p(K)}\big\|_{l^q}.\]
	 The proposition follows.
\end{proof}

\subsection{Graph Laplacians and a tent function decomposition}
Now, we turn to the case when $\sigma>\mathscr{L}_1(p)$. In this case, we have $B^{p,q}_\sigma(K)\subset C(K)$ as a well-known result{\cite{HZ}}, and we would expect this to happen for $\Lambda^{p,q}_\sigma(K)$. This can be easily seen from the following Lemma.

\begin{lemma}\label{lemma45}
Let $1<p<\infty$ and $f\in L^p(K)$. we define $E_{B_t(x)}(f)=\frac1{\mu(B_t(x))}\int_{B_t(x)}fd\mu$ for any $x\in K$ and $0<t\leq 1$. Then we have 

(a). If $\{x_j\}_{j=1}^n\subset K$ is a finite set of points such that $\sum_{j=1}^n 1_{B_{rt}(x_j)}\leq \lambda$ for some finite number $\lambda<\infty$, then 
\[\big(\sum_{j=1}^n |E_{B_t(x_j)}(f)-E_{B_{rt}(x_j)}(f)|^p\big)^{1/p}\lesssim \lambda^{1/p} t^{-d_H/p}I_p\big(f,(r+1)t\big).\]

(b). For $1\leq q\leq \infty$, $\Lambda^{p,q}_\sigma(K)\subset C(K)$ for $\sigma>\frac{d_S}{p}$.
\end{lemma}
\begin{proof}
(a) follows from a direct estimate,
\[\begin{aligned}
&\big(\sum_{j=1}^n |E_{B_t(x_j)}(f)-E_{B_{rt}(x_j)}(f)|^p\big)^{1/p}\\
\lesssim& t^{-d_H/p}\big(\sum_{j=1}^n\int_{B_{rt}(x_j)}|f(y)-E_{B_t(x_j)}(f)|^pd\mu(y)\big)^{1/p}\\
\lesssim& t^{-d_H/p}\Big(\sum_{j=1}^n\int_{B_{rt}(x_j)}\big((r+1)t\big)^{-d_H}\int_{B_{(r+1)t}(y)}|f(y)-f(z)|^pd\mu(z)d\mu(y)\Big)^{1/p}\\
\leq &\lambda^{1/p}t^{-d_H/p}I_p\big(f,(r+1)t\big).
\end{aligned}\]

(b). Let $f\in \Lambda^{p,q}_\sigma(K)$ and notice that $\frac{\sigma d_W}{2}> \frac{d_H}{p}$.
First, by (a), we can see that $f(x)=\lim_{m\to\infty} E_{B_{r^m}(x)}(f)$ is well defined for every $x\in K$, and 
\[\big|f(x)-E_{B_{r^m}(x)}(f)\big|\leq C_1\sum_{n=m}^\infty r^{-nd_H/p}I_p(f,r^n)\leq C_2r^{m(\sigma d_W/2-d_H/p)}\|f\|_{\Lambda^{p,\infty}_\sigma(K)}.\]
Next, we fix $x,y\in K$ such that $R(x,y)<r^m$. Choose $k>0$ such that $r^k<1/2$ and let $m'=(m-k)\vee 0$. It is not hard to see that 
\[\big|E_{B_{r^m}(x)}(f)-E_{B_{r^m}(y)}(f)\big|\leq C_3r^{-md_H/p}I_p(f,r^{m'})\leq C_4r^{m(\sigma d_W/2-d_H/p)}\|f\|_{\Lambda^{p,\infty}_\sigma(K)}.\]
Combining the above two estimates, we can see that $f$ is continuous.
\end{proof}

In the above proof, we actually see that $f$ is H\"{o}lder continuous on $K$ if $f\in \Lambda^{p,q}_\sigma(K)$ with $\sigma>\mathscr L_1(p)$. In addition, we will have a characterization of $\Lambda^{p,q}_\sigma(K)$ based on the discrete Laplacian on $\Lambda_m$ for $\mathscr{L}_1(p)<\sigma<\mathscr{C}(p)$.

\begin{definition}
(a). For $m\geq 0$, define the {\em graph energy form} on $V_{\Lambda_m}$ by \[\mathcal{E}_{\Lambda_m}(f,g)=\sum_{w\in \Lambda_m}r_w^{-1}\mathcal{E}_0(f\circ F_w,g\circ F_w), \quad \forall f,g\in l(V_{\Lambda_m}).\] 

(b). Define $H_{\Lambda_m}: l(V_{\Lambda_m})\to l(V_{\Lambda_m})$ the {\em graph Laplacian} associated with $\mathcal{E}_{\Lambda_m}$, i.e. 
\[\mathcal{E}_{\Lambda_m}(f,g)=-<H_{\Lambda_m}f,g>_{l^2(V_{\Lambda_m})}=-<f,H_{\Lambda_m}g>_{l^2(V_{\Lambda_m})}, \quad \forall f,g\in l(V_{\Lambda_m}),\]
where $l^2(V_{\Lambda_m})$ stands for the discrete $l^2$ inner product over $V_{\Lambda_m}$ with the counting measure. 
 
(c). For $m\geq 1$, define $J_m=\big\{f\in C(K): f \text{ is harmonic in } F_wK,  \forall w\in \Lambda_m, \text{ and }f|_{V_{\Lambda_{m-1}}}=0\big\}$, and call $J_m$ the space of {\em level-$m$ tent functions}. For convenience, set $J_0=\mathcal{H}_0$.

(d). For $1<p<\infty$, $1\leq q\leq \infty$ and $\sigma>\frac{d_S}{p}$, define $$\Lambda^{p,q}_{\sigma,(1)}(K)=\big\{f\in C(K):\{r^{m(-\sigma d_W/2+1+d_H/p)}\|H_{\Lambda_m}f\|_{l^p(V_{\Lambda_{m}})}\}\in l^q\big\},$$ 
with norm $\|f\|_{\Lambda^{p,q}_{\sigma,(1)}(K)}=\|f\|_{L^p(K)}+\big\|r^{m(-\sigma d_W/2+1+d_H/p)}\|H_{\Lambda_m}f\|_{l^p(V_{\Lambda_m})}\big\|_{l^q}$.
\end{definition}

For each $f\in C(K)$, clearly $f$ admits a unique expansion in terms of tent functions $f=\sum_{m=0}^\infty f_m$ with $f_m\in J_m,\forall m\geq 0$. 

Before proceeding, let's first collect some easy observations.

\begin{lemma}\label{lemma47}
Let $1<p<\infty$ and $u\in J_m$ with $m\geq 0$.

(a). $H_{\Lambda_n}u|_{\mathring{V}_{\Lambda_{n}}}=0$ for any $n>m$. 

(b). $\|u\|_{L^p(K)}\asymp r^{m(1+d_H/p)}\|H_{\Lambda_m}u\|_{l^p(\mathring{V}_{\Lambda_{m}})}\asymp r^{m(1+d_H/p)}\|H_{\Lambda_m}u\|_{l^p({V}_{\Lambda_{m}})}$.

(c). For any $\sigma<\mathscr{C}(p)$, we have $I_{p}(u,r^n)\leq Cr^{(n-m)\sigma d_W/2}\|u\|_{L^p(K)}$ for all $n\geq m$. 
\end{lemma}
\begin{proof}
(a) is trivial since $u$ is harmonic on $\mathring{V}_{\Lambda_n}$ by definition. 	

(b) can be seen following the estimate
\[\begin{aligned}
\|u\|_{L^p(K)}&=\big(\sum_{w\in \Lambda_{m}}\mu_w\|u\circ F_w\|^p_{L^p(K)}\big)^{1/p}\asymp r^{md_H/p}\big(\sum_{w\in \Lambda_{m}}\|u\circ F_w\|^p_{L^p(K)}\big)^{1/p}\\
&\asymp r^{md_H/p}\big(\sum_{x\in \mathring{V}_{\Lambda_m}}|u(x)|^p\big)^{1/p}.
\end{aligned}\]

(c). For $m=0$, the result trivially follows from the definition of $\mathscr C(p)$. It suffices to consider $m\geq 1$ case. Choose $k$ such that $R(x,y)>r^{k+n}$ for any $x,y$ not in adjacent cells in $\{F_{\tau} K: {\tau}\in\Lambda_n\}$ and for any $n\geq 0$.  Now, for any  fixed $w\in \Lambda_m$, we consider the integral
\[\big(\int_{F_wK}r^{-nd_H}\int_{B_{r^{n}}(x)}|u(x)-u(y)|^pd\mu(y)d\mu(x)\big)^{1/p}\leq I(1,w)+I(2,w)\]
with
\[I(1,w):=(\int_{F_wK}r^{-nd_H}\int_{B_{r^{n}}(x)\cap F_wK}|u(x)-u(y)|^pd\mu(y)d\mu(x)\big)^{1/p},\]
\[I(2,w):=(\int_{F_wK}r^{-nd_H}\int_{B_{r^{n}}(x)\setminus F_wK}|u(x)-u(y)|^pd\mu(y)d\mu(x)\big)^{1/p}.\]

By the definition of $\mathscr{C}(p)$, noticing that $u$ is harmonic in $F_wK$, we have 
\[I(1,w)\leq Cr^{(n-m)\sigma d_W/2}\|u\|_{L^p(F_wK)}.\]
For $I(2,w)$, we can see that in fact $B_{r^n}(x)\setminus F_wK\neq \emptyset$ only if $x$ stays in a cell $F_\tau K$ with $\tau\in\Lambda_{n-k}$ which contains a point $z\in F_wV_0$. Without loss of generality, we assume that $n-k\geq m$, and  sum $I(2,w)$'s over $\Lambda_m$ to get
\[\begin{aligned}
\big(\sum_{w\in\Lambda_m} I(2,w)^p\big)^{1/p}&\leq \Big(\sum_{\tau,\tau'}\int_{x\in F_\tau K}\int_{y\in F_{\tau'}K}r^{-nd_H}|u(x)-u(y)|^pd\mu(y)d\mu(x)\Big)^{1/p}\\
&\leq C\big(\sum_{\tau}\sum_{z\in V_{\Lambda_m}\cap F_\tau K}\int_{x\in F_\tau K}|u(x)-u(z)|^pd\mu(x)\big)^{1/p}\\
&\leq C'\big(r^{nd_H}\sum_{w \in \Lambda_m} r^{(n-m)p}\|u\circ F_w\|^p_{L^p(K)}\big)^{1/p}\\
&=C'r^{(n-m)(1+d_H/p)}\big(\sum_{w \in \Lambda_m}\|u\|^p_{L^p(F_wK)}\big)^{1/p},
\end{aligned}\]
where $\tau,\tau'\in \{\tau''\in \Lambda_{n-k}:F_{\tau''}K\cap V_{\Lambda_m}\neq \emptyset\}$, and we require $\tau\neq \tau'$ with $F_\tau K\cap F_{\tau'}K\neq \emptyset$ in the first line.

Combining the estimates on $I(1,w)$'s and $I(2,w)$'s, and noticing that \[1+\frac{d_H}{p}\geq \mathscr{C}(p)\frac{d_W}{2}>\frac{\sigma d_W}{2}\]
by Proposition \ref{prop33}, (c) follows.  
\end{proof}

Now, we state the main result in this subsection.

\begin{theorem}\label{thm48}
Let $f\in C(K)$ with $f=\sum_{m=0}^\infty f_m$ and $f_m\in J_m,\forall m\geq 0$. For $1<p<\infty$, $1\leq q\leq \infty$,   for the claims

$(1)\text{ }f\in \Lambda^{p,q}_\sigma(K)$; $(2)\text{ }f\in \Lambda^{p,q}_{\sigma,(1)}(K)$; $(3)\text{ }\big\{r^{-m\sigma d_W/2}\|f_m\|_{L^p(K)}\big\}_{m\geq 0}\in l^q$,

\noindent we can say:

(a). If $\sigma>\mathscr{L}_1(p)$, we have $(1)\Longrightarrow (2)\Longrightarrow (3)$, with \[\|f\|_{\Lambda^{p,q}_\sigma(K)}\gtrsim \|f\|_{\Lambda^{p,q}_{\sigma,(1)}(K)}\gtrsim \big\|r^{-m\sigma d_W/2}\|f_m\|_{L^p(K)}\big\|_{l^q}.\] 

(b). If $\mathscr{L}_1(p)<\sigma<\mathscr{L}_2(p)$, we have $(1)\Longrightarrow (2)\iff (3)$ with 
\[\|f\|_{\Lambda^{p,q}_\sigma(K)}\gtrsim \|f\|_{\Lambda^{p,q}_{\sigma,(1)}(K)}\asymp \big\|r^{-m\sigma d_W/2}\|f_m\|_{L^p(K)}\big\|_{l^q}.\] 

(c). If $(\frac{1}{p},\sigma)\in \mathscr{A}_1$, we have $(1)\iff (2)\iff (3)$ with 
\[\|f\|_{\Lambda^{p,q}_\sigma(K)}\asymp \|f\|_{\Lambda^{p,q}_{\sigma,(1)}(K)}\asymp \big\|r^{-m\sigma d_W/2}\|f_m\|_{L^p(K)}\big\|_{l^q}.\] 
\end{theorem}
\begin{proof}
(a). We first prove $(1)\Longrightarrow (2)$. We follow the conventional notation to denote $x\sim_m y$ if $x,y\in F_wV_0$ for some $w\in \Lambda_m$. We fix $k>0$ so that $B_{r^m}(y)\subset B_{r^{m-k}}(x)$ for any $x\sim_m y$ and $m\geq 0$. Then we can see 
\[\big(\sum_{x\sim_m y}|E_{B_{r^m}(x)}(f)-E_{B_{r^m}(y)}(f)|^p\big)^{1/p}\lesssim r^{-md_H/p}I_p(f,r^{m-k}).\]
Since each vertex $x\in V_{\Lambda_m}$ is of bounded degree, by writing \[f(x)=E_{B_{r^m}(x)}(f)+\sum_{n=m}^\infty\big(E_{B_{r^{n+1}}(x)}(f)-E_{B_{r^n}(x)}(f)\big),\] we  apply Lemma \ref{lemma45} (a) to see that 
\[\big(\sum_{x\sim_m y}|f(x)-f(y)|^p\big)^{1/p}\lesssim r^{-md_H/p}I_p(f,r^{m-k})+\sum_{n=m}^\infty r^{-nd_H/p}I_p(f,r^n).\]
Write $\tilde{H}_{\Lambda_m}f=r^{m(1+d_H/p)}H_{\Lambda_m}f$ for convenience. We then have
\[
\begin{aligned}
\|\tilde{H}_{\Lambda_m}f\|_{l^p(V_{\Lambda_m})}&\lesssim  I_p(f,r^{m-k})+\sum_{n=m}^\infty r^{(m-n)d_H/p}I_p(f,r^n)\\
&= {I_p(f,r^{m-k})+}\sum_{n=0}^\infty r^{-nd_H/p}I_p(f,r^{m+n}).
\end{aligned}
\]
Noticing that 
\[\begin{aligned}
\big\|r^{-m\sigma d_W/2}\sum_{n=0}^\infty r^{-nd_H/p}I_p(f,r^{m+n})\big\|_{l^q}&=\big\|\sum_{n=0}^\infty r^{n(\sigma d_W/2-d_H/p)}r^{-(m+n)\sigma d_W/2}I_p(f,r^{m+n})\big\|_{l^q}\\
&\leq \sum_{n=0}^\infty r^{n(\sigma d_W/2-d_H/p)}\cdot\big\|r^{-m\sigma d_W/2}I_p(f,r^{m})\big\|_{l^q}\\
&\lesssim \big\|r^{-m\sigma d_W/2}I_p(f,r^{m})\big\|_{l^q},
\end{aligned}\]
since we assume $\sigma>\mathscr{L}_1(p)=\frac{d_S}{p}$, the claim follows.\vspace{0.2cm}

$(2)\Longrightarrow (3)$ is easy. We can see that $H_{\Lambda_m}\big(\sum_{n=0}^{m-1}f_m\big)|_{\mathring{V}_{\Lambda_m}}\equiv 0$ by Lemma \ref{lemma47} (a), and $(\sum_{n=0}^m f_m)|_{V_{\Lambda_m}}=f|_{V_{\Lambda_m}}$. Thus $H_{\Lambda_m }f|_{\mathring{V}_{\Lambda_m}}=H_{\Lambda_m}f_m|_{\mathring{V}_{\Lambda_m}}$ and then
\[\|f_m\|_{L^p(K)}\asymp \|\tilde{H}_{\Lambda_m}{f}\|_{l^p(\mathring{V}_{\Lambda_m})}\]
{using Lemma \ref{lemma47} (b)}. The claim follows immediately. \vspace{0.2cm}

(b). It remains to show $(3)\Longrightarrow (2)$. By the definition of $\mathcal{E}_{\Lambda_m}$ and the fact that $f_m$ is harmonic in $F_wK$ for each $w\in \Lambda_m$, we can see that 
\[\|H_{\Lambda_n}f_m\|_{l^p(V_{\Lambda_n})}=\|H_{\Lambda_m}f_m\|_{l^p(V_{\Lambda_m})}\asymp r^{-m(1+d_H/p)}\|f_m\|_{L^p(K)},\forall n\geq m.\]
Also, $\|H_{\Lambda_n}f_m\|_{l^p(V_{\Lambda_n})}=0$ for $n<m$. Thus, 
\[\begin{aligned}
\big\|r^{m(-\sigma d_W/2+1+d_H/p)}\|H_{\Lambda_m}f\|_{l^p(V_{\Lambda_m})}\big\|_{l^q}&\lesssim  \big\|\sum_{n=0}^mr^{(m-n)(1+d_H/p)-m\sigma d_W/2}\|f_n\|_{L^p(K)}\big\|_{l^q}\\
&=\big\|\sum_{n=0}^mr^{n(1+d_H/p)-m\sigma d_W/2}\|f_{m-n}\|_{L^p(K)}\big\|_{l^q}\\
&\lesssim \sum_{n=0}^\infty r^{n(1+d_H/p-\sigma d_W/2)}\cdot\big\|r^{-m\sigma d_W/2}\|f_m\|_{L^p(K)}\big\|_{l^q}.
\end{aligned}\]
The claim follows since $1+\frac{d_H}{p}-\frac{\sigma d_W}{2}>0$ by $\sigma<\mathscr L_2(p)$.\vspace{0.2cm}

(c). We only need to prove $(3)\Longrightarrow (1)$ when $\sigma<\mathscr{C}(p)$. We fix $\sigma<\eta<\mathscr{C}(p)$, and by using Lemma \ref{lemma47} (c) we can see 
\[I_p(f_m,r^n)\lesssim r^{(n-m)\eta d_W/2}\|f_m\|_{L^p(K)},\forall n\geq m\geq 0.\]
The claim then follows by applying the above estimate and Lemma \ref{lemma43} (b) to get similar estimates as (\ref{eqn43}), (\ref{eqn44}) and (\ref{eqn45}), using a same proof as the second half of Proposition \ref{prop44}. 
\end{proof}

We end this section with the following theorem, whose proof will be completed in Section 5 and Section 6. 

\begin{theorem}\label{thm49}
	For $1<p<\infty$, $1\leq q\leq \infty$ and $\mathscr{L}_1(p)<\sigma<2$, we have  $B^{p,q}_\sigma(K)=\Lambda^{p,q}_{\sigma,(1)}(K)$ with $\|\cdot\|_{\Lambda^{p,q}_{\sigma,(1)}(K)}\asymp \|\cdot\|_{B^{p,q}_\sigma(K)}$.
\end{theorem}

\section{Embedding $\Lambda^{p,q}_\sigma(K)$ into $B^{p,q}_\sigma(K)$}
In this section, we will use the $J$-method of real interpolation to prove that $\Lambda^{p,q}_\sigma(K)\subset B^{p,q}_\sigma(K)$ for $0<\sigma<2$. This will not involve the critical curve $\mathscr{C}$.

We will use the following fact about the real interpolation, which is obvious from the $J$-method. Readers may find details of the $J$-method in the book \cite{interpolation}.

\begin{lemma}\label{lemma51}
	Let $\bar{X}:=(X_0,X_1)$ be an interpolation couple of Banach spaces, $1\leq q\leq \infty$, $0<\theta<1$, and $\bar{X}_{\theta,q}:=(X_0,X_1)_{\theta,q}$ be the real interpolation space. For  each $x\in X_0\cap X_1$ and $\lambda>0$, write \[J(\lambda,x)=\max\big\{\|x\|_{X_0},\lambda\|x\|_{X_1}\big\}.\]
	If $\{g_m\}_{m\geq 0}\subset X_0\cap X_1$ is a sequence with $\big\{\lambda^{-m\theta}J(\lambda^{m},g_m)\big\}_{m\geq 0}\in l^q$ and $\lambda>0$, then 
	\[g=\sum_{m=0}^\infty g_m\in \bar X_{\theta,q}\]
	with $\|g\|_{\bar{X}_{\theta,q}}\lesssim \big\|\lambda^{-m\theta}J(\lambda^{m},g_m)\big\|_{l^q}$.
\end{lemma}

We start with the following lemma.
\begin{lemma}\label{lemma52}
	For $1<p<\infty$, $f\in L^p(K)$ and $m\geq 0$, there is a function $u_m$ in $H^p_2(K)$ such that \[E[u_m|\Lambda_m]=E[f|\Lambda_m],\]
	and in addition, 
	\[\begin{cases}
	\big\|u_m-E[f|\Lambda_m]\big\|_{L^p(K)}\leq CI_{p}(f,r^{m-k}),\\
	\|\Delta u_m\|_{L^p(K)}\leq Cr^{-md_W}I_{p}(f,r^{m-k}),
	\end{cases}\]
where $k\in \mathbb{N}$ and $C>0$ are constants independent of $f$ and $m$.
\end{lemma}
\begin{proof}
For convenience, we write $E[f|\Lambda_m]=\sum_{w\in \Lambda_m} c_w1_{F_wK}$ with $c_w\in\mathbb{C}$. 

For each $x\in V_{\Lambda_m}$ and $n\geq m$, we write 
$U_{x,n}=\bigcup\{F_{w'}K: x\in F_{w'}K, w'\in \Lambda_{n}\}$,
and take $m'\geq m$ to be the smallest one, such that 
\[\# V_{\Lambda_m}\cap U_{x,m'}\leq 1, \quad U_{x,m'}\cap U_{y,m'}=\emptyset, \quad \forall x,y\in V_{\Lambda_{m}}.\]
Clearly, the difference $m'-m$ is bounded for all $m$.

Let $U_{m'}=\bigcup_{x\in V_{\Lambda_m}}U_{x,m'}$. We define $u_m$ as follows. 

\noindent 1). For $x\in K\setminus U_{m'}$, we define $u_m(x)=E[f|\Lambda_m](x)$.

\noindent 2). For $x\in V_{\Lambda_m}$, we let $M_x=\#\{w\in \Lambda_m:x\in F_wK\}$ and define 
\[u_m(x)=\frac{1}{\# M_x}\sum_{w\in\Lambda_m, F_wK\ni x}c_w.\]

\noindent 3). It remains to construct $u_m$ on each $U_{x,m'}$ with $x\in V_{\Lambda_m}$. In this case, for each $F_{w'}K\subset U_{x,m'}$ with $w'\in \Lambda_{m'}$, we have already defined its boundary values. For $u_m$ in  $F_{w'}K$, additionally we require that $u_m$ satisfies the Neumann boundary condition on $F_{w'}V_0$, and $E_{w'}(u_m)=c_w$ for $w$ to be the word in $\Lambda_m$ such that $F_{w'}K\subset F_wK$. It is easy to see the existence of such a function locally on $F_{w'}K$, and the following estimate can be achieved by scaling,
\[\begin{cases}
\|u_m-c_w\|_{L^p(F_{w'}K)}\lesssim r^{md_H/p}\big|u_m(x)-c_w\big|,\\
\|\Delta u_m\|_{L^p(F_{w'}K)}\lesssim r^{m(d_H/p-d_W)}\big|u_m(x)-c_w\big|.
\end{cases}\]

With (1),(2) and (3), we obtain a function $u_m\in H^p_2(K)$ such that $E[u_m|\Lambda_m]=E[f|\Lambda_m]$. It remains to show the desired estimates for $u_m$. First, we have 
\[\begin{aligned}
\big\|u_m-E[f|\Lambda_m]\big\|_{L^p(K)}&=\big(\sum_{w'}\|u_m-c_w\|^p_{L^p(F_{w'}K)}\big)^{1/p}\\
&\lesssim \big(r^{md_H}\sum_{w'}|u_m(x_{w'})-c_w|^p\big)^{1/p}
\lesssim  \big(r^{md_H}\sum_{w\sim_m \nu}|c_w-c_\nu|^p\big)^{1/p},
\end{aligned}\]
where the summation $\sum_{w'}$ is over all $w'\in\Lambda_{m'}$ such that $F_{w'}K\cap V_{\Lambda_{m}}\neq \emptyset$ and $x_{w'}$ is the single vertex in $F_{w'}K\cap V_{\Lambda_m}$, $w$ stands for the word in $\Lambda_m$ such that $F_{w'}K\subset F_{w}K$, and  $\sum_{w\sim_m\nu}$ is over all the pairs $w,\nu\in \Lambda_m$ with $F_wK\cap F_\nu K\neq \emptyset$.  By choosing $k\in\mathbb{N}$ such that $r^{m-k}>2\diam F_wK$ for any $w\in \Lambda_m$ (clearly this $k$ can be chosen to work for all $m$), we then have 
\[\big(r^{md_H}\sum_{w\sim \nu}|c_w-c_\nu|^p\big)^{1/p}\lesssim I_p(f,r^{m-k}),\]
thus we get the first desired estimate.
The estimate for $\|\Delta u_m\|_{L^p(K)}$ is essentially the same.
\end{proof}

Now, we prove the main result of this section.
\begin{proposition}\label{prop53}
For $1<p<\infty$, $1\leq q\leq \infty$ and $0<\sigma<2$, we have $\Lambda^{p,q}_\sigma(K)\subset B^{p,q}_\sigma(K)$ with $\|\cdot\|_{B^{p,q}_\sigma(K)}\lesssim \|\cdot\|_{\Lambda^{p,q}_\sigma(K)}$.
\end{proposition}

\begin{proof}
	Let $f\in \Lambda^{p,q}_\sigma(K)$. We define a sequence of functions $u_m$ in $H^p_2(K)$ by Lemma \ref{lemma52}, and we take 
	\[g_m=\begin{cases}
	u_0,&\text{ if }m=0,\\
	u_m-u_{m-1},&\text{ if }m>0.
	\end{cases}\]
	For $m\geq 0$, by Lemma \ref{lemma52}, we have the estimate 
	\[\begin{cases}
	\|g_m\|_{L^p(K)}\lesssim \big\|\tilde{E}[f|\Lambda_m]\big\|_{L^p(K)}+I_p(f,r^{m-k})+I_p(f,r^{m-k-1})\lesssim I_p(f,r^{m-k-1}),\\
	\|\Delta g_m\|_{L^p(K)}\lesssim r^{-md_W} \big(I_p(f,r^{m-k})+I_p(f,r^{m-k-1})\big)\lesssim r^{-md_W}I_p(f,r^{m-k-1}),
	\end{cases}\]
	where $k$ is the same as Lemma \ref{lemma52}, and we use Lemma \ref{lemma43} (a) in the second estimate of the first formula. Taking $\lambda=r^{d_W}$, $X_0=L^p(K)$ and $X_1=H^p_2(K)$ in Lemma \ref{lemma51}, it then follows that 
	\[\begin{aligned}
    \big\|r^{-m\sigma d_W/2}J(r^{md_W},g_m)\big\|_{l^q}&\leq \big\|r^{-m\sigma d_W/2}\|g_m\|_{L^p(K)}\big\|_{l^q}+\big\|r^{m(1-\sigma/2)d_W}\|g_m\|_{H^p_2(K)}\big\|_{l^q}\\
    &\lesssim \big\|r^{-m\sigma d_W/2}I_p(f,r^{m-k-1})\big\|_{l^q}\lesssim\|f\|_{\Lambda^{p,q}_\sigma(K)}.
    \end{aligned}\]
    It is easy to see that $f=\sum_{m=0}^\infty g_m$, so combining with Lemma \ref{lemma26},  we have $f\in B^{p,q}_\sigma(K)$ with $\|f\|_{B^{p,q}_\sigma(K)}\lesssim \|f\|_{\Lambda^{p,q}_{\sigma}(K)}$.
\end{proof}

Before ending this section, we mention that the same method can be applied to show that $\Lambda^{p,q}_{\sigma,(1)}(K)\subset B^{p,q}_\sigma(K)$ for $\mathscr{L}_1(p)<\sigma<2$. In this case, for each $m\geq 0$, we choose a piecewise harmonic function of level $m$ that coincides with $f$ at $V_{\Lambda_m}$, then modify it in a neighbourhood of $V_{\Lambda_m}$ to get a function $u_m$ in $H^p_2(K)$ analogous to that in Lemma \ref{lemma52}. Since the idea is essentially the same, we omit the proof, and state the result as follows.

\begin{proposition}\label{prop54}
	For $1<p<\infty$, $1\leq q\leq \infty$ and $\mathscr{L}_1(p)<\sigma<2$, we have $\Lambda^{p,q}_{\sigma,(1)}(K)\subset B^{p,q}_\sigma(K)$ with $\|\cdot\|_{B^{p,q}_\sigma(K)}\lesssim \|\cdot\|_{\Lambda^{p,q}_{\sigma,(1)}(K)}$.
\end{proposition}

\section{Embedding $B^{p,q}_\sigma(K)$ into $\Lambda^{p,q}_\sigma(K)$}
We will prove $B^{p,q}_\sigma(K)\subset \Lambda^{p,q}_\sigma(K)$ for $0<\sigma<\mathscr C(p)$ in this section. Also, we will show that $B^{p,q}_\sigma(K)\subset \Lambda^{p,q}_{\sigma,(1)}(K)$ for $\mathscr{L}_1(p)<\sigma<2$. 

First, let's look at two easy lemmas.

\begin{lemma}\label{lemma61}
	Let $\big\{X_m,\|\cdot\|_{X_m}\big\}_{m\geq 0}$ be a sequence of Banach spaces. For $1\leq q\leq\infty$ and $\alpha>0$, let 
	\[l^q_\alpha(X_\cdot)=\big\{\bm{s}=\{s_m\}_{m\geq 0}:s_m\in X_m,\forall m\geq 0,\text{ and }\{\alpha^{-m}\|s_m\|_{X_m}\}\in l^q\big\},\]
	be the space with norm $\|\bm{s}\|_{l^q_\alpha(X_\cdot)}=\big\|\alpha^{-m}\|s_m\|_{X_m}\big\|_{l^q}$. Then, for $\alpha_0\neq \alpha_1$, $0<\theta<1$ and $1\leq q_0,q_1,q\leq \infty$, we have 
	\[\big(l^{q_0}_{\alpha_0}(X_\cdot),l^{q_1}_{\alpha_0}(X_\cdot)\big)_{\theta,q}=l^{q}_{\alpha_\theta}(X_\cdot), \text{ with } \alpha_\theta=\alpha_0^{(1-\theta)}\alpha_1^\theta.\]
\end{lemma}

This lemma is revised from Theorem 5.6.1 in book \cite{interpolation} with a same argument. The difference  is that we allow each coordinate taking values in different spaces, which does not bring any difficult to the proof.  As an immediate consequence, we can see the following interpolation lemma for $\Lambda^{p,q}_\sigma(K)$ and $\Lambda^{p,q}_{\sigma,(1)}(K)$. 

\begin{lemma}\label{lemma62}
Let $1\leq p,q,q_0,q_1\leq\infty$, $0<\sigma_0,\sigma_1<\infty$ and $0<\theta<1$. We have 

(a). $\big(\Lambda^{p,q_0}_{\sigma_0}(K),\Lambda^{p,q_1}_{\sigma_1}(K)\big)_{\theta,q}\subset \Lambda^{p,q}_{\sigma_\theta}(K)$;

(b). $\big(\Lambda^{p,q_0}_{\sigma_0,(1)}(K),\Lambda^{p,q_1}_{\sigma_1,(1)}(K)\big)_{\theta,q}\subset \Lambda^{p,q}_{\sigma_\theta,(1)}(K)$,

\noindent with $\sigma_\theta=(1-\theta)\sigma_0+\theta\sigma_1$.
\end{lemma}

From now on, we will separate our consideration into two cases, according to $(\frac{1}{p},\sigma)$ belongs to $\mathscr A_1$ or $\mathscr A_2$. We will deal with the border between $\mathscr{A}_1$ and $\mathscr{A}_2$ by using Lemma \ref{lemma62}. For the region $\mathscr A_1$, in fact, we mainly consider a larger region $\mathscr B:=\big\{(\frac 1 p,\sigma): \mathscr L_1(p)<\sigma<\mathscr L_2(p)\big\}$ instead.

\subsection{On regions $\mathscr{A}_1$ and $\mathscr B$}
To reach the goal that $B^{p,q}_\sigma(K)\subset \Lambda^{p,q}_\sigma(K)$ for $(\frac 1 p,\sigma)\in\mathscr A_1$, by Theorem \ref{thm48} (c), it suffices to prove $B^{p,q}_\sigma(K)\subset\Lambda^{p,q}_{\sigma,(1)}(K)$.
We will fulfill this for the parameter region $\mathscr B=\big\{(\frac 1 p,\sigma): \mathscr L_1(p)<\sigma<\mathscr L_2(p)\big\}$, which is the region between the two critical lines $\mathscr L_1$ and $\mathscr L_2$, and of cause contains $\mathscr A_1$. 

Note that we can write each $f\in C(K)$ as a unique series,
\[f=\sum_{m=0}^\infty f_m, \text{ with } f_m\in J_m, \quad\forall m\geq 0,\]
and in addition by Theorem \ref{thm48} (b), it always holds
\[\|f\|_{\Lambda^{p,q}_{\sigma,(1)}(K)}\asymp \big\|r^{-m\sigma d_W/2}\|f_m\|_{L^p(K)}\big\|_{l^q}.\]

Let's start with the following observation.

\begin{lemma}\label{lemma63}
Let $1<p,q<\infty$ and $0<\sigma<2$. Write $p'=\frac{p}{p-1}$ and $q'=\frac{q}{q-1}$. Then there is a continuous quadratic form $\tilde{\mathcal{E}}(\cdot,\cdot)$ on $B^{p,q}_{\sigma}(K)\times B^{p',q'}_{2-\sigma}(K)$ such that
\[\tilde{\mathcal{E}}(f,g)=\mathcal{E}(f,g),\quad\forall f\in B^{p,q}_\sigma(K)\cap dom\mathcal{E},\quad g\in H^{p'}_{2}(K).\]	
\end{lemma}
\begin{proof} First, by the definition of $\Delta_N$, we can see that 
\begin{equation}\label{eqn61}
\big|\mathcal{E}(f,g)\big|=\big|\int_K f\Delta_N gd\mu\big|\leq \|f\|_{L^p(K)}\|g\|_{H^{p'}_{2}(K)},
\end{equation}
for any $f\in L^p(K)\cap dom\mathcal{E}$ and $g\in H^{p'}_{2}(K)$. So there is a continuous quadratic form $\tilde{\mathcal{E}}:L^p(K)\times H^{p'}_{2}(K)\to\mathbb{C}$,  such that $\tilde{\mathcal{E}}(f,g)=\mathcal{E}(f,g)$ for any $f\in L^p(K)\cap dom\mathcal{E}$ and $g\in H^{p'}_2(K)$. 
In addition, we can see that $\mathcal{E}(f,g)\leq \|f\|_{H^p_2(K)}\cdot \|g\|_{L^{p'}(K)}$ for any $f\in H^p_2(K)$ and $g\in H_2^{p'}(K)$.  

As a consequence, the mapping $f\to\tilde{\mathcal{E}}(f,\cdot)$ is continuous from $L^p(K)$ to $\big(H^{p'}_{2}(K)\big)^*$, and is continuous from $H^p_{2}(K)$ to $\big(L^{p'}(K)\big)^*$ since $H^{p'}_{2}(K)$ is dense in $L^{p'}(K)$, where we use $*$ to denote the dual space. Using the theorem of real interpolation (See {\cite{interpolation}} Theorem 3.7.1), we have $f\to\tilde{\mathcal{E}}(f,\cdot)$ is continuous from $B^{p,q}_{\sigma}(K)$ to $\big(B^{p',q'}_{2-\sigma}(K)\big)^*$. So {$\tilde{\mathcal{E}}$} extends to a continuous quadratic form on $B^{p,q}_{\sigma}(K)\times B^{p',q'}_{2-\sigma}(K)$.
\end{proof}

\begin{proposition}\label{prop64}
For $1<p<\infty$, $1\leq q\leq \infty$ and $\mathscr{L}_1(p)<\sigma<\mathscr{L}_2(p)$, we have  $B^{p,q}_\sigma(K)=\Lambda^{p,q}_{\sigma,(1)}(K)$ with $\|\cdot\|_{\Lambda^{p,q}_{\sigma,(1)}(K)}\asymp \|\cdot\|_{B^{p,q}_\sigma(K)}$.
In particular, if $(\frac 1p,\sigma)\in \mathscr A_1$, we have $B^{p,q}_\sigma(K)=\Lambda^{p,q}_\sigma(K)$ with $\|\cdot\|_{\Lambda^{p,q}_\sigma(K)}\asymp\|\cdot\|_{B^{p,q}_\sigma(K)}.$
\end{proposition}
\begin{proof} 
By Theorem \ref{thm48} (c), it suffices to prove the first result. Also, by Lemma \ref{lemma62} (b) and Proposition \ref{prop54}, it suffices to consider the $1<q<\infty$ case. 

By Lemma \ref{lemma63}, there exists $\tilde{\mathcal{E}}$ on $B^{p,q}_\sigma(K)\times B^{p',q'}_{2-\sigma}(K)$ with $p'=\frac{p}{p-1},q'=\frac{q}{q-1}$ satisfying (\ref{eqn61}). In the following claims, we provide an exact formula of $\tilde{\mathcal{E}}$ on $\Lambda^{p,q}_{\sigma,(1)}(K)\times \Lambda^{p',q'}_{2-\sigma,(1)}(K)$. \vspace{0.2cm}

\noindent\textit{Claim 1. Let $M\geq 0$,  $f=\sum_{m=0}^M f_m$ and $g=\sum_{m=0}^M g_m$ with $f_m,g_m\in J_m$, $0\leq m\leq M$. We have 
\[\tilde{\mathcal{E}}(f,g)=\sum_{m=0}^M \mathcal{E}(f_m,g_m)=-\sum_{m=0}^M <f_m,H_{\Lambda_m}g_m>_{l^2(V_{\Lambda_m})}.\]}

\textit{Proof.} Clearly, we have $f\in dom\mathcal{E}\cap B^{p,q}_\sigma(K)$, and $g\in B^{p',q'}_{2-\sigma}(K)$ by  Proposition \ref{prop54}, thus there is a sequence of functions $g^{(n)}$ in $ H^{p'}_2(K)$ converging to $g$ in $B^{p',q'}_{2-\sigma}(K)$. For each $n$, by Lemma \ref{lemma63}, we have 
\[\tilde{\mathcal{E}}(f,g^{(n)})=\mathcal{E}(f,g^{(n)})=\sum_{m=1}^M\mathcal{E}(f_m,g^{(n)})=-\sum_{m=1}^M<H_{\Lambda_m}f_m,g^{(n)}>_{l^2(V_{\Lambda_m})}.\]
Letting $n\rightarrow\infty$, we have the claim proved, since $g^{(n)}$ converges to $g$ in $B^{p',q'}_{2-\sigma}(K)$ and thus converges uniformly as $2-\sigma>\mathscr{L}_1(p')$. \hfill$\square$\vspace{0.2cm}

\noindent\textit{Claim 2. Let $f=\sum_{m=0}^\infty f_m$ and $g=\sum_{m=0}^\infty g_m$ with $f_m,g_m\in J_m,\forall m\geq 0$, and
\[\big\|r^{-m\sigma d_W/2}\|f_m\|_{L^p(K)}\big\|_{l^q}<\infty,\quad \big\|r^{m(\sigma -2)d_W/2}\|g_m\|_{L^{p'}(K)}\big\|_{l^{q'}}<\infty.\]
We have 
\[\tilde{\mathcal{E}}(f,g)=-\sum_{m=0}^\infty <H_{\Lambda_m}f_m,g_m>_{l^2(V_{\Lambda_m})}.\]}

\textit{Proof.} By using Claim 1, we have  $\tilde{\mathcal{E}}\big(\sum_{m=0}^Mf_m,\sum_{m=0}^Mg_m\big)=-\sum_{m=0}^M<H_{\Lambda_m}f_m,g_m>_{l^2(V_{\Lambda_m})}$ for any $M\geq 0$. Letting $M\rightarrow\infty$, then the claim follows, since the left side converges to $\tilde{\mathcal{E}}(f,g)$ as $\sum_{m=0}^M f_m$ converges to $f$ in $B^{p,q}_\sigma(K)$ and  $\sum_{m=0}^M g_m$ converges to $g$ in $B^{p',q'}_{2-\sigma}(K)$ by Proposition \ref{prop54}.\hfill$\square$\vspace{0.2cm}

\noindent\textit{Claim 3. Let $f=\sum_{m=0}^\infty f_m$ with $f_m\in J_m,\forall m\geq 0$, and $\big\|r^{-m\sigma d_W/2}\|f_m\|_{L^p(K)}\big\|_{l^q}<\infty$. We have $\big\|r^{-m\sigma d_W/2}\|f_m\|_{L^p(K)}\big\|_{l^q}\lesssim \|f\|_{B^{p,q}_\sigma(K)}$}. \vspace{0.2cm}

\textit{Proof.} The space $l^{q'}\big(l^{p'}(V_{\Lambda_\cdot})\big)$ can be identified with the dual space of $l^{q}\big(l^{p}(V_{\Lambda_\cdot})\big)$ in a natural way, and thus we can find $g=\sum_{m=0}^\infty g_m$, with $g_m\in J_m$ and $0<\big\|r^{m(\sigma -2)d_W/2}\|g_m\|_{L^{p'}(K)}\big\|_{l^{q'}}<\infty$, such that
\[\begin{aligned}
\big|\tilde{\mathcal{E}}(f,g)\big|&\geq \frac{1}{2}\big\|r^{-m\sigma d_W/2+m+md_H/p}\|H_{\Lambda_m}f_m\|_{l^p({V}_{\Lambda_m})}\big\|_{l^q}\cdot\big\|r^{m(\sigma-2)d_W/2+md_H/{p'}}\|g_m\|_{l^{p'}({V}_{\Lambda_m})}\big\|_{l^{q'}}\\
&\gtrsim\big\|r^{-m\sigma d_W/2}\|f_m\|_{L^p(K)}\big\|_{l^q}\cdot \big\|r^{m(\sigma -2)d_W/2}\|g_m\|_{L^{p'}(K)}\big\|_{l^{q'}}\\
&\gtrsim\big\|r^{-m\sigma d_W/2}\|f_m\|_{L^p(K)}\big\|_{l^q}\cdot \|g\|_{B^{p',q'}_{2-\sigma}(K)}. 
\end{aligned}\] 
On the other hand, we have $\big|\tilde{\mathcal{E}}(f,g)\big|\lesssim \|f\|_{B^{p,q}_\sigma(K)}\cdot\|g\|_{B^{p',q'}_{2-\sigma}(K)}$. The estimate follows.\hfill$\square$\vspace{0.2cm}

Now, combining Claim 3 and Proposition \ref{prop54}, we can see that $\Lambda^{p,q}_{\sigma,(1)}(K)$ is a closed subset of $B^{p,q}_\sigma(K)$. On the other hand, we have $H^p_2(K)\subset \Lambda^{p,q}_{\sigma,(1)}(K)$ by an easy estimate. So the desired result follows since  $H^p_2(K)$ is dense in $B^{p,q}_\sigma(K)$. For $q=1,\infty$, the result follows simply by Lemma \ref{lemma62} (b).
\end{proof}

By applying Proposition \ref{prop54}, Lemma \ref{lemma62} (b) and Proposition \ref{prop64}, we can finish the proof of Theorem \ref{thm49}.

\begin{proof}[Proof of Theorem \ref{thm49}]
It is easy to see that $H^p_2(K)\subset \Lambda^{p,\infty}_{2,(1)}(K)$. In fact, for each $f\in H^p_2(K)$ and $x\in V_{\Lambda_m}$, we have 
\[H_{\Lambda_m}f(x)=\int_{U_{x,m}}\psi_{x,m}(\Delta f)d\mu,\]
where $U_{x,m}$ is the same we defined in the proof of Lemma \ref{lemma52}, and $\psi_{x,m}$ is a piecewise harmonic function supported on $U_{x,m}$, with $\psi_{x,m}(x)=1$ and $\psi_{x,m}|_{V_{\Lambda_m}\setminus \{x\}}\equiv 0$, and is harmonic in each $F_wK,w\in \Lambda_m$.  As a consequence, we get \[r^{-md_H/p'}\|H_{\Lambda_m}f\|_{l^p(V_{\Lambda_m})}\lesssim \|\Delta f\|_{L^p(K)},\]
which yields that  $H^p_2(K)\subset \Lambda^{p,\infty}_{2,(1)}(K)$. 

Then by applying Lemma \ref{lemma62} (b) and Proposition \ref{prop64}, using real interpolation, we can see that $B^{p,q}_\sigma(K)\subset \Lambda^{p,q}_{\sigma,(1)}(K)$ for $\mathscr{L}_1(p)<\sigma<2$. Combining this with Proposition \ref{prop54}, the theorem follows.
\end{proof}

\subsection{On region $\mathscr{A}_2$}
It remains to show $B^{p,q}_\sigma(K)\subset \Lambda^{p,q}_\sigma(K)$ on $\mathscr{A}_2$. In fact, by Proposition \ref{prop64} and Lemma \ref{lemma62} (a), noticing that $L^p(K)$ is contained in ``$\Lambda^{p,\infty}_0(K)$'', we can simply cover a large portion of $\mathscr{A}_2$, see an illustration in Figure \ref{triarea}. However, it remains unclear for the strip region near $p=1$, if $\mathscr{C}$ and $\mathscr{L}_1$ intersect at some point with $p>1$. We will apply another idea to overcome this. Also, we mention here that a similar method can solve the $\mathscr{A}_1$ region as well with necessary modifications. 

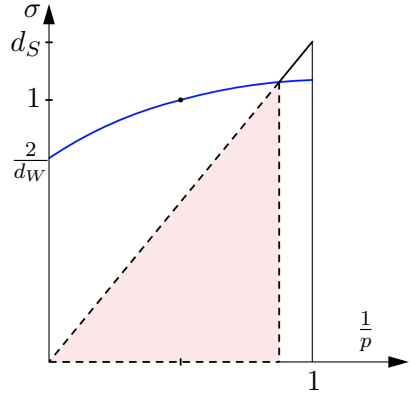
\begin{figure}[h]
	\includegraphics[width=5cm]{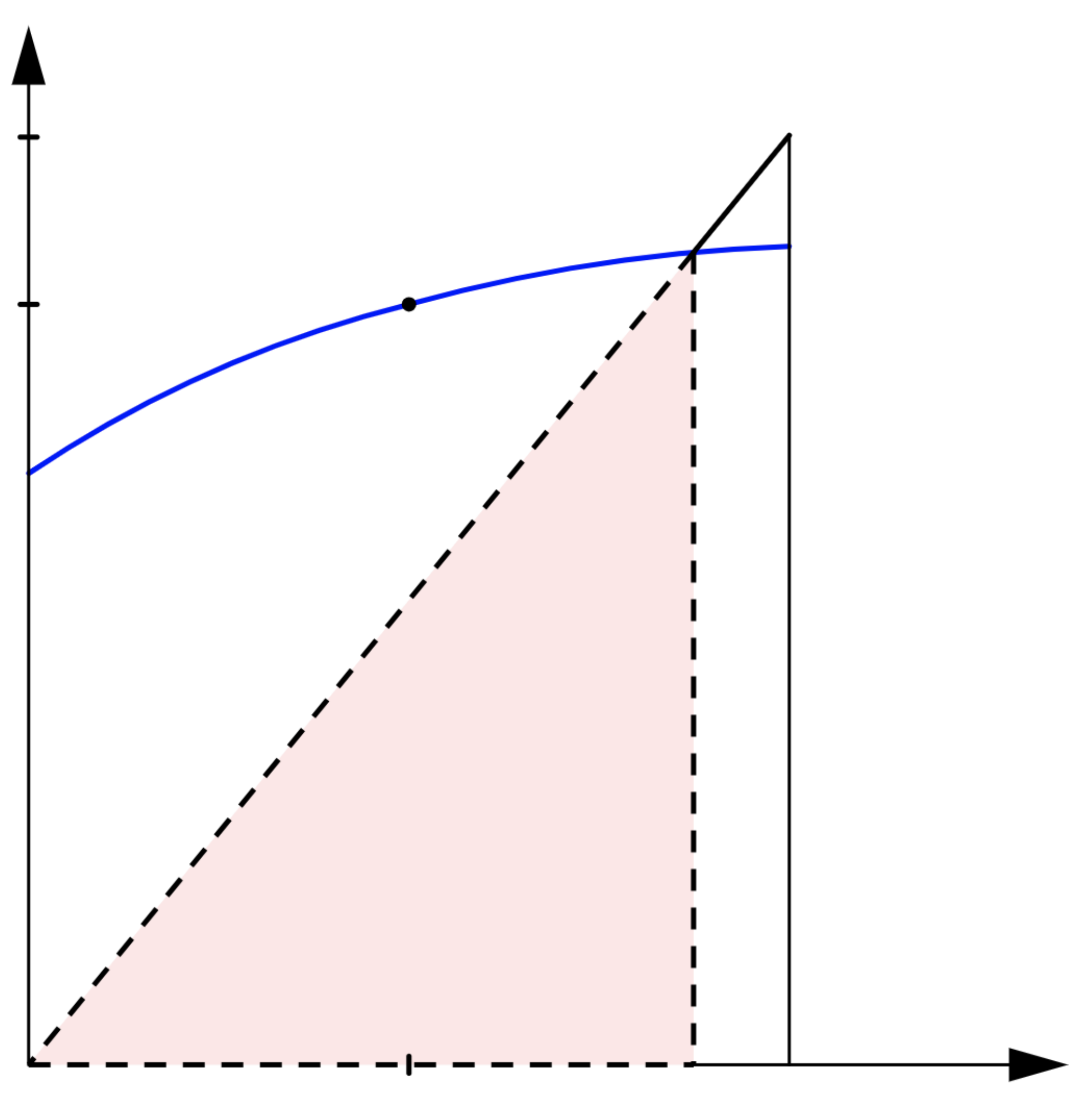}
	\begin{picture}(0,0)
	\put(-152,135){$\sigma$}
	\put(-151,101){$1$}
	\put(-157,122){$d_S$}
	\put(-157,78){$\frac{2}{d_W}$}
	\put(-26,14){$\frac 1p$}
	\put(-45,-6){$1$}
	\end{picture}
	\caption{A portion of $\mathscr A_2$.}
	\label{triarea}
\end{figure}

We will rely on Proposition \ref{prop44} in this part, which says, for $0<\sigma<\mathscr L_1(p)$, it holds that 
\[\|f\|_{\Lambda^{p,q}_\sigma(K)}\asymp \big\|r^{-m\sigma d_W/2}\|\tilde{E}[f|\Lambda_m]\|_{L^p(K)}\big\|_{l^q}.\]
To get a reasonable estimate for $\big\|\tilde{E}[f|\Lambda_m]\big\|_{L^p(K)}$, we start with a new decomposition. 

\begin{definition}\label{def66}
(a). For $m\geq 0$, we define $T_m=\big\{\sum_{w\in \Lambda_m}h_w\circ F_w^{-1}:h_w\in \mathcal{H}_0,\forall w\in \Lambda_m\big\}$.

(b). Write $P_{T_m}$ for the orthogonal projection $L^2(K)\to T_m$ for  $m\geq 0$, and 
	\[
	{P}_{\tilde T_m}=\begin{cases}
	P_{T_0},&\text{ if }m=0,\\
	P_{T_m}-P_{T_{m-1}},&\text{ if }m\geq 1.
	\end{cases}
	\]
Clearly, $P_{T_m}$ extends naturally to $L^p(K)\to T_m$	for any $1\leq p\leq \infty$. 
	
(c). Write $\tilde{T}_m=\big\{P_{\tilde T_m}f:f\in L^1(K)\big\}$ for $m\geq 1$, and write $\tilde{T}_0=T_0$.  
\end{definition}
\noindent\textbf{Remark.} The spaces $T_m$ are collections of piecewise harmonic functions, but may not be continuous at $V_{\Lambda_m}\setminus V_0$. \vspace{0.2cm}

We collect some useful results in the following lemma.

\begin{lemma}\label{lemma67}
Let $1<p<\infty$, $m\geq 0$ and $u\in \tilde{T}_m$.

(a). We have $\tilde{E}[u|\Lambda_n]=0$ if $n<m$.

(b). For any $0<\sigma<\mathscr{C}(p)$, we have $\big\|\tilde{E}[u|\Lambda_n]\big\|_{L^p(K)}\lesssim r^{(n-m)\sigma d_W/2}\|u\|_{L^p(K)}$ for $n\geq m$.
\end{lemma}
\begin{proof}
(a). By definition, for each $u\in \tilde{T}_m$, we have $P_{T_{m-1}}u=0$. On the other hand, we can see that $\oplus_{l=0}^n \tilde{J}_l\subset T_{m-1}$ since clearly $\tilde{J}_l$ consists of piecewise constant functions. Thus, 
\[\tilde{E}[u|\Lambda_n]=\tilde{E}[P_{T_{m-1}}u|\Lambda_n]=0, \quad \forall n<m.\]

(b). We first look at $m=0$ case. By definition of $\mathscr{C}$, we have $r^{-n\sigma d_W/2}I_p(u,r^n)\lesssim\|u\|_{L^p(K)}$, as $u\in \tilde{T}_m=\mathcal{H}_0$. The claim then follows by applying Lemma \ref{lemma43} (a). 

For general case, for each $w\in\Lambda_m$, we can see that \[r^{-(n-m)\sigma d_W/2}\big\|(E[u|\Lambda_n])\circ F_w\big\|_{L^p(K)}\lesssim\|u\circ F_w\|_{L^p(K)},\quad \forall n\geq m.\]  (b) then follows by scaling and summing the estimates over $\Lambda_m$.
\end{proof}

\begin{proposition}\label{prop68}
	For $1<p<\infty$, $1\leq q\leq \infty$ and  $(\frac{1}{p},\sigma)\in \mathscr{A}_2$, we have $B^{p,q}_\sigma(K)\subset \Lambda^{p,q}_\sigma(K)$ with $\|\cdot\|_{\Lambda^{p,q}_\sigma(K)}\lesssim\|\cdot\|_{B^{p,q}_\sigma(K)}$. 
\end{proposition}
\begin{proof}
Let $f\in B^{p,q}_\sigma(K)$, it suffices to show that $\big\|r^{-m\sigma d_W/2}\|\tilde{E}[f|\Lambda_m]\|_{L^p(K)}\big\|_{l^q}\lesssim \|f\|_{B^{p,q}_\sigma(K)}$ by applying Proposition \ref{prop44}. 

For convenience, we write $\mathscr{P}:L^p(K)\to \prod_{m=0}^\infty \tilde{T}_m$, defined as $\mathscr{P}(f)_m=P_{\tilde T_m}f$. Also, we equip each $\tilde{T}_m$ with the $L^p$ norm. It is obvious that $P_{T_m}$ is from $L^p(K)$ to $L^p(K)$, so is $P_{\tilde T_m}$. We have the following claims.
\vspace{0.2cm}
 
\noindent\textit{Claim 1. $\mathscr{P}$ is bounded from $L^p(K)$ to $l^\infty(\tilde{T}_\cdot)$.}\vspace{0.2cm}

\noindent\textit{Claim 2. $\mathscr{P}$ is bounded from $H^p_2(K)$ to $l^\infty_{r^{d_W}}(\tilde{T}_\cdot)$.}

\textit{Proof.} Let $G$ be the Green's operator on $K$ \cite{ki3,s3}. For any $f\in H^p_2(K)=H^p_{2,D}(K)\oplus \mathcal{H}_0$, we have 
\[\|f-P_{T_0}f\|_{L^p(K)}\lesssim \|G\Delta f\|_{L^p(K)}\lesssim \|\Delta f\|_{L^p(K)},\]
where the first inequality is due to the fact that $f-G(-\Delta) f\in T_0=\mathcal{H}_0$, and the second inequality is due to the fact that $G$ is bounded from $L^p(K)$ to $L^p(K)$. We apply the above estimate locally on each $F_wK$ with $w\in \Lambda_m$ to get 
\[\|f-P_{T_m}f\|_{L^p(K)}\lesssim r^{md_W}\|\Delta f\|_{L^p(K)},\]
by using the scaling property of $\Delta f$. Thus, we have \[\|{P}_{\tilde T_m}f\|_{L^p(K)}\lesssim \|f-P_{ T_m}f\|_{L^p(K)}+ \|f-P_{ T_{m-1}}f\|_{L^p(K)}\lesssim r^{md_W}\|\Delta f\|_{L^p(K)}.\]
This finishes the proof of Claim 2.\hfill$\square$\vspace{0.2cm}

Combining Claim 1 and Claim 2, and using Lemma \ref{lemma61}, we see the following claim.\vspace{0.2cm}

\noindent\textit{Claim 3. For $1<p<\infty$, $1\leq q\leq \infty$ and $0<\sigma<2$,  $\mathscr{P}$ is bounded from $B^{p,q}_\sigma(K)$ to $l^q_{r^{\sigma d_W/2}}(\tilde{T}_\cdot)$. }\vspace{0.2cm}

Now we turn to the proof of the proposition. We fix a parameter point $(\frac 1p, \sigma)$ in $\mathscr A_2$. By Claim 3, we can see that for each $f\in B^{p,q}_\sigma(K)$ we clearly have $f=\sum_{m=0}^\infty P_{\tilde T_m}f$, with the series absolute convergent in $L^p(K)$. Thus, we have 
\[\tilde{E}[f|\Lambda_m]=\sum_{n=0}^\infty \tilde{E}[P_{\tilde T_n}f|\Lambda_m]=\sum_{n=0}^m \tilde{E}[P_{\tilde T_n}f|\Lambda_m],\]
where the second equality is due to Lemma \ref{lemma67} (a). In addition, by applying Lemma \ref{lemma67} (b), we have the estimate
\[\big\|\tilde{E}[f|\Lambda_m]\big\|_{L^p(K)}\lesssim \sum_{n=0}^m r^{(m-n)\eta d_W/2}\|P_{\tilde T_n}f\|_{L^p(K)},\]
where $\eta$ is a fixed number such that $\sigma<\eta<\mathscr{C}(p)$. As a consequence, we then have 
\[\begin{aligned}
\big\|r^{-m\sigma d_W/2}\|\tilde{E}[f|\Lambda_m]\|_{L^p(K)}\big\|_{l^q}&\lesssim \big\|r^{-m\sigma d_W/2}\sum_{n=0}^m r^{(m-n)\eta d_W/2}\|P_{\tilde T_n}f\|_{L^p(K)}\big\|_{l^q}\\
&= \big\|r^{-m\sigma d_W/2}\sum_{n=0}^m r^{n\eta d_W/2}\|P_{\tilde T_{m-n}}f\|_{L^p(K)}\big\|_{l^q}\\
&\leq \sum_{n=0}^\infty r^{n(\eta-\sigma)d_W/2}\cdot\big\|r^{-m\sigma d_W/2}\|P_{\tilde T_m}f\|_{L^p(K)}\big\|_{l^q}\lesssim \|f\|_{B^{p,q}_\sigma(K)},
\end{aligned}\]
where we use Claim 3 again in the last inequality.
\end{proof}

\noindent\textbf{Remark.} We can apply a similar argument as Lemma \ref{lemma67} and Proposition \ref{prop68} for $(\frac 1p,\sigma)\in 
\mathscr B$ to show that $B^{p,q}_\sigma(K)\subset \Lambda^{p,q}_{\sigma,(1)}(K)$, as stated in Proposition \ref{prop64}. The difference is that $f_m\in J_m$ in the tent function expansion of $f=\sum_{m=0}^\infty f_m$ depends on $P_{\tilde T_n}f$ for $n\geq m$. This gives a second proof of Proposition $\ref{prop64}$.  \vspace{0.2cm}

We finish this section with a conclusion that Theorem \ref{thm11} holds.\vspace{0.2cm}

\noindent\textit{Proof of Theorem \ref{thm11}.}  On $\mathscr{A}_1$, the theorem follows from Proposition \ref{prop64}; on $\mathscr{A}_2$, the theorem follows from Proposition \ref{prop53} and Proposition \ref{prop68}; lastly, on the border between $\mathscr{A}_1$ and $\mathscr{A}_2$, we  have $B^{p,q}_\sigma(K)\subset \Lambda^{p,q}_\sigma(K)$ by interpolation using Lemma \ref{lemma62} (a), as well as the other direction is covered by Proposition \ref{prop53}. \hfill$\square$

\section*{Acknowledgments}
The authors are grateful  to Lijian Yang for providing the figures.

\bibliographystyle{amsplain}

\end{document}